\documentclass[12pt,a4paper,reqno]{amsart}
\usepackage{amsmath,amscd,amstext,amsthm,amsfonts,latexsym}
\usepackage[dvips]{graphicx}
\usepackage{graphicx, graphics}
\usepackage{amsmath}
\usepackage{color}
\usepackage[active]{srcltx}
\usepackage{a4wide}
\usepackage{amssymb, amsmath, mathtools}
\usepackage{graphicx}
\usepackage{float}
\usepackage[active]{srcltx}
\usepackage{hyperref}

\theoremstyle{plain}
\newtheorem{theorem}{Theorem}[section]
\newtheorem{proposition}[theorem]{Proposition}
\newtheorem{lemma}[theorem]{Lemma}
\newtheorem{corollary}[theorem]{Corollary}
\newtheorem{maintheorem}{Theorem}

\theoremstyle{definition}
\newtheorem{remark}[theorem]{Remark}
\newtheorem{example}[theorem]{Example}
\newtheorem{definition}[theorem]{Definition}

\newcommand{\field}[1]{\mathbb{#1}}

\newcommand{\NN}{\field{N}}

\newcommand{\vep}{\delta}

\newcommand{\cF}{{\mathcal F}}

\newcommand{\Leb}{\operatorname{{Leb}}}

\title[Semigroup actions of expanding maps]{Semigroup actions of expanding maps}

\begin{document}

\author[M. Carvalho]{Maria Carvalho}
\address{CMUP \& Departamento de Matem\'atica, Universidade do Porto, Portugal.}
\email{mpcarval@fc.up.pt}

\author[F. Rodrigues]{Fagner B. Rodrigues}
\address{Departamento de Matem\'atica, Universidade Federal do Rio Grande do Sul, Brazil.}
\email{fagnerbernardini@gmail.com}

\author[P.Varandas]{Paulo Varandas}
\address{Departamento de Matem\'atica, Universidade Federal da Bahia, Brazil.}
\email{paulo.varandas@ufba.br}



\maketitle
\begin{abstract}
We consider semigroups of Ruelle-expanding maps, parameterized by random walks on the free semigroup, with the aim of examining their complexity and exploring the relation between intrinsic properties of the semigroup action and the thermodynamic formalism of the associated skew-product. In particular, we clarify the connection between the topological entropy of the semigroup action and the growth rate of the periodic points, establish the main properties of the dynamical zeta function of the semigroup action and prove the existence of stationary probability measures.
\end{abstract}

\section{Introduction}

In the mid seventies the thermodynamic formalism was brought from statistical mechanics to dynamical systems by the pioneering work of Sinai, Ruelle and Bowen~\cite{Ru1, Bowen, Si72}. The correspondence between one-dimensional lattices and uniformly hyperbolic dynamics conveyed several notions from one setting to the other, introducing, via Markov partitions, Gibbs measures and equilibrium states into the realm of dynamical systems; see, for instance, \cite{Bo75}. 
Within non-invertible dynamics, a complete description of the thermodynamic formalism has been established for Ruelle-expanding maps~\cite{Ru2} and for expansive maps with a specification property~\cite{Ru92, HR}. In particular, it is known that for every potential under some regularity condition there exists a unique equilibrium state, 
which is a Gibbs measure 
and has exponential decay of correlations.

The classical strategy to prove these properties ultimately relies on the analysis of the spectral properties of the Perron-Fr\"obenius transfer operator, and we may extend this method to finitely generated group actions. Yet, the attempts to generalize the previous results have so far been riddled with difficulties, and a global theory is still unknown. Some success has been registered within continuous actions of finitely generated abelian groups. More precisely, the statistical mechanics of expansive $\mathbb{Z}^d$-actions satisfying a specification property has been studied by Ruelle in \cite{Ru73}, after introducing a suitable notion of pressure 
and discussing its link with measure theoretical entropy and free energy. 
The crucial ingredient in this context has been the fact that continuous $\mathbb{Z}^d$-actions on compact spaces admit probability measures invariant under every continuous map involved in the group action. With it, Ruelle proved a variational principle for the topological pressure and built equilibrium states as the class of pressure maximizing invariant probability measures. This duality between topological and measure theoretical complexity of the dynamical system has been later used by Eizenberg, Kifer and Weiss~\cite{EKW} to establish large deviations principles for $\mathbb{Z}^d$-actions satisfying a specification property.

A unified approach to the thermodynamic formalism for continuous group actions in the absence of probability measures invariant under all elements of the group is still incomplete. Although these actions are not dynamical systems, a few definitions of topological pressure have been proposed, although most of them unrelated and assuming either abelianity, amenability or some growth rate of the corresponding group. Inspired by the notion of complexity presented by Bufetov in \cite{Bufetov}, in the context of skew-products, where no commutativity or conditions on the semigroup growth rate are required, the second and third named authors introduced in  \cite{RoVa1} a notion of topological pressure for a semigroup action and showed that it reflects the complex behavior of the action. 
In the present paper we push this analysis further, considering semigroups of Ruelle-expanding maps. We relate the notion of topological entropy of a semigroup action introduced in \cite{RoVa1} with the growth of the set of periodic orbits, the concepts of fibered and relative entropies and the radius of convergence of a dynamical zeta function for the semigroup action. Moreover, we generalize the classical Ruelle-Perron-Fr\"obenius transfer operator and construct stationary measures and equilibrium states for semigroup actions of $C^2$ expanding maps. 
Meanwhile, we examine how the choice of the random walk in the semigroup unsettles the ergodic properties of the action.

\section{Setting}\label{sec:setting}

Let $M$ be a compact metric space and $C^0(M)$ denote the space of all continuous observable functions $\psi: M \to \mathbb R$. Given a finite set of continuous maps $g_i:M \to M$, $i \in \{1,2,\ldots,p\}$, $p\ge 1$, and the finitely generated semigroup $(G,\,\circ)$ with the finite set of generators $G_1=\{id, g_1, g_2, \dots, g_p\}$, we will write
$$G=\bigcup_{n\in \mathbb N_0} G_n$$
where $G_0=\{id\}$ and $\underline g \in G_n$ if and only if $\underline g=g_{i_n} \dots g_{i_2} g_{i_1}$, with $g_{i_j} \in G_1$ (for notational simplicity's sake we will use $g_j \, g_i$ instead of the composition $g_j\,\circ\, g_i$). A semigroup can have multiple generating sets. We will assume that the generator set $G_1$ is minimal, meaning that no function $g_j$, for $j = 1, \ldots, p$, can be expressed as a composition from the remaining generators.
\subsubsection*{\emph{\textbf{Free semigroups}}} Observe that each element $\underline g$ of $G_n$ may be seen as a word that originates from the concatenation of $n$ elements in $G_1$. Clearly, different concatenations may generate the same element in $G$. Nevertheless, in most of the computations to be done, we shall consider different concatenations instead of the elements in $G$ they create. One way to interpret this statement is to consider the itinerary map
$$
\begin{array}{rccc}
\iota : & F_p & \to & G \\
	& \underline i=i_n \dots i_1 & \mapsto & \underline g_{\underline i} := g_{i_n} \dots g_{i_1}
\end{array}
$$
where $F_p$ is the free semigroup with $p$ generators, and to regard concatenations on $G$ as images by $\iota$ of paths on $F_p$.

Set $G_1^* =G_1 \setminus \{id\}$ and, for every $n\ge 1$, let $G_n^*$ denote the space of concatenations of $n$ elements in $G_1^*$. To summon each element $\underline{g}$ of $G^*_n$, we will write $|\underline{g}|=n$ instead of $\underline g\in G^*_n$.
In $G$, one consider the semigroup operation of concatenation defined as usual: if $\underline{g}=g_{i_n} \dots g_{i_2} g_{i_1}$ and $\underline{h}=h_{i_m} \dots h_{i_2} h_{i_1}$, where $n=|\underline g|$ and $m=|\underline h|$, then
$$
\underline{g}\,\underline{h}=g_{i_n} \dots g_{i_2} g_{i_1} h_{i_m} \dots h_{i_2} h_{i_1} \in G_{m+n}^*.
$$
We then say that the finitely generated semigroup $G$ induces a \emph{semigroup action} $S: G \times M \to M$ in $M$ if, for any $\underline{g},\,\underline{h}  \in G$ and all $x \in M$, we have $S(\underline{g}\,\underline{h},x)=S(\underline{g}, S(\underline{h},x)).$ The action $S$ is said to be continuous if, for any $\underline g \in G$, the map $\underline g : M \to M$ given by $\underline g(x) = S(\underline{g},x)$ is continuous. As usual, $x\in M$ is said to be a \emph{fixed point} for $\underline g \in G$ if $\underline g(x)=x$; the set of these fixed points will be denoted by $\text{Fix}(\underline g)$. A point $x \in M$ is said to be a \emph{periodic point of period $n$} if there exists $\underline{g} \in G_n^*$ such that $\underline{g}(x) = x$. We let $\text{Per}(G_n)= \bigcup_{|\underline{g}|=n} \,\text{Fix}(\underline{g})$ denote the set of all periodic points of period $n$. Accordingly, $\text{Per}(G)=\bigcup_{n\ge 1} \text{Per}(G_n)$ will stand for the set of periodic points of the whole semigroup action. We observe that, when $G_1^*=\{f\}$, this definition coincides with the usual one of periodic points for the dynamical system $f$.

\subsubsection*{\emph{\textbf{Ruelle-expanding maps}}}
Let $(X,d)$ be a compact metric space and $T:X \rightarrow X$ a continuous map. $T$ is said to be \emph{expansive} if there exist $\varepsilon >0$ such that, for any $x,y \in X$,
$$\sup_{n\in \mathbb{N}}\,\,d(T^n(x),T^n(y))<\varepsilon  \quad \Rightarrow  \quad x=y.$$
The map $T$ is \emph{locally expanding} if there exist $\lambda>1$ and $\delta>0$ such that, for any $x,y \in X$,
$$d(x,y)<\delta \quad \Rightarrow \quad d(T(x),T(y))> \lambda \,d(x,y).$$
These two notions are bonded: every locally expanding system is expansive and an expansive map is locally expanding with respect to an adapted metric (cf. \cite{CR}).

\begin{definition}\label{de.Ruelle-expanding}
The system $(X,T)$ is \emph{Ruelle-expanding} if $T$ is locally expanding and open, that is:
\begin{enumerate}
\item There exists $c>0$ such that, for all $x, y \in X$ with $x\neq y$, we have
$$T(x)=T(y) \quad \Rightarrow \quad d(x,y)>c.$$
\item There are $r>0$ and $0<\rho <1$ such that, for each $x \in X$ and all $a \in T^{-1}(\{x\})$ there is a map $\varphi: B_r(x) \rightarrow X$, defined on the open ball centered at $x$ with radius $r$, such that $\varphi(x)=a$, $T\circ \varphi (z)=z$ and, for all $z,w \in B_r(x)$, we have
$$d(\varphi(z),\varphi(w))\leq \rho \,d(z,w).$$
\end{enumerate}
\end{definition}

In Section~\ref{sec:Ruelle-expanding}, we will recall the most relevant topological and ergodic properties of Ruelle-expanding maps. (The reader acquainted with these results may omit this section.)

\section{Statement of the main results}\label{sec:statements}

We start with a topological description of finitely generated semigroups of uniformly expanding maps. Later, we will begin assuming that $G_1$ is either a finite subset of Ruelle-expanding maps acting on a compact connected metric space $M$ or a finite subset of the space $End^2(M)$ of non-singular $C^2$ endomorphisms in a compact connected Riemannian manifold $M$. In this setting, we will show that the set of periodic points with period $n$ for such a semigroup dynamics has a definite exponential growth rate with the period $n$, which is given by the topological entropy of the semigroup (see Definition~\ref{de.top_entropy}) and this is equal to the logarithm of the spectral radius of a suitable choice of Ruelle-Perron-Fr\"obenius transfer operators $(\mathbf{L}_{n,0})_{n\ge 1}$ for the semigroup action (we refer the reader to Subsections~\ref{sec:RPFoperators} and \ref{sec:topol-entropy} for the precise definitions of these concepts).

\begin{maintheorem}\label{thm:A}
Let $G$ be the semigroup generated by a set $G_1=\{Id,g_1,\dots,g_p\}$, where $G_1^*$ is a set of Ruelle-expanding maps on a compact connected metric space $M$ and let $S:G\times M \to M$ be its continuous semigroup action. Then
$$
0<h_{\text{top}}(S) = \lim_{n\to\infty} \frac{1}{n}\log \Big(  \frac{1}{p^n} \sum_{|\underline g|=n} \sharp\, \text{Fix}(\underline g) \Big) = sp\,\left((\mathbf{L}_{n,0})_{n\in \mathbb{N}}\right).
$$
\end{maintheorem}

The second part of this work concerns the asymptotic growth of the periodic points of a semigroup action. Assume that the periodic points for a continuous mapping $f:M \to M$ are isolated. Then, following \cite{AM}, the Artin-Mazur zeta function $\zeta_f$ of $f$ is defined as
$$\zeta_f(z) = exp\left(\sum_{n=1}^{+\infty}\,\frac{\sharp \,\text{Per}_n(f)}{n}\,z^n\right)$$
with $z$ a complex variable and $\sharp\,\text{Per}_n(f)$ standing for the the number of periodic points of $f$ with period $n$. In a similar way, given a semigroup $G$ generated by a set $G_1=\{g_1,\dots,g_p\}$ of maps on a Riemannian manifold $M$
and the corresponding semigroup action $S : G \times M \to M$, we define the zeta function of $G$ by
$$z \in \mathbb C \quad \mapsto \quad \zeta_S(z) = \exp \left(\sum_{n=1}^\infty \,\frac{N_n(G)}{n}\, z^n \right)$$
where
$$N_n(G)=\frac{1}{p^n}\sum_{|\underline g|=n}\, \sharp \, \text{Fix}(\underline g).$$
The natural queries about this function are motivated by Subsection~\ref{zeta function}. There, we will recall that, under suitable conditions, the sequence $(N_n(G))_{n \in \mathbb{N}}$ has a definite rate of growth
$$\wp(S)=\limsup_{n\to + \infty} \,\frac{1}{n}\log \, (\max\{N_n(G),1\})$$
and that the function $\zeta_S$ is rational, and so has a meromorphic extension to the complex plane. In general, what kind of function is $\zeta_S$? How is its radius of convergence related with the topological entropy of the semigroup action?

\begin{maintheorem}\label{thm:B}
Let $G$ be the semigroup generated by a set $G_1=\{Id,g_1,\dots,g_p\}$, where $G_1^*$ is a set of Ruelle-expanding maps on a compact connected
metric space $M$ and $S: G\times M \to M$ the corresponding continuous semigroup action on $M$. Then the dynamical zeta function $\zeta_S$ is rational and its radius of convergence is $\rho_S=e^{-\wp(S)}=e^{-h_{\text{top}}(S)}.$
\end{maintheorem}

Afterwards, we study the ergodic properties of semigroups of maps in $End^2(M)$. Although one does not expect to find an absolutely continuous common invariant probability measure, we may hope to discover some probability measure which reflects an averaged distribution of the Lebesgue measure under the action of the semigroup. We say that $R_{\underline \theta}$ is a \emph{random walk on $G$} if $R_{\underline \theta}=\iota_* \theta^{\mathbb N}$, where $\theta$ is a probability measure on $\iota^{-1}(G_1^*)=\{1, \dots, p\}$, the set of generators of $F_p$. For instance, if $\theta (i)=\frac{1}{p}$ for any $i \in \{1, \dots, p\}$, then $\eta_{\underline p}=\theta^{\mathbb N}$ is the equally distributed Bernoulli probability measure on the Borel sets of the unilateral shift $\Sigma^+_p=\{1, \dots, p\}^{\mathbb N}$. If, instead, $\theta (i)=a_i>0$ for each $i \in \{1, \dots, p\}$, then $\underline a = (a_1,a_2,\ldots,a_p)$ is a non-trivial probability vector and $\eta_{\underline a}=\theta^{\mathbb N}$ will stand for the Bernoulli probability measure $\theta^{\mathbb N}$ on $\Sigma^+_p$, while $R_{\underline a}=\iota_* (\eta_{\underline a})$ will denote the corresponding random walk $R_{\underline \theta}$ on $G$. More generally, we may take a $\sigma$-invariant probability measure $\eta$ on $\Sigma^+_p$ and consider the associated random walk $R_{\eta}=\iota_* (\eta)$. Given a $\sigma$-invariant probability measure $\eta$, a probability measure $\nu$ on $M$ is said to be $R_{\eta}$-\emph{stationary} if
$$\nu  = \int \underline g_* \nu \; dR_{\eta}(\underline g)$$
that is, for every continuous observable $\psi : M \to \mathbb R$, we have
\begin{equation}\label{de.stationary}
\int \psi \; d\nu = \int \Big( \int \psi \circ \underline g \, d \nu \Big) \; dR_{\eta}(\underline g).
\end{equation}
We will show that every semigroup action of $C^{2}$ expanding maps admits a unique stationary measure which is absolutely continuous with respect to the Lebesgue measure and whose density can be obtained via the iterations of a Ruelle-Perron-Fr\"obenius operator. Finally, we discuss a thermodynamic formalism 
and find 
an adequate definition of measure of maximal entropy for a semigroup action with respect to a fixed random walk.

\begin{maintheorem}\label{thm:C}
Let $G$ be the semigroup generated by a set $G_1=\{id, g_1,\dots,g_p\}$ of $C^2$ expanding maps on a compact connected Riemannian manifold $M$. Consider the corresponding continuous semigroup action $S: G\times M \to M$ and the random walk $R_{\underline p}=\iota_*(\eta_{\underline p})$. Then, the semigroup action has:
\begin{itemize}
\item an absolutely continuous stationary probability measure $\nu_1$;
\item a probability measure of maximal entropy $\nu_0$.
\end{itemize}
Moreover, the maximal entropy measure can be computed as the weak$^{\ast}$-limit of an averaged distribution of either preimages or periodic points.
\end{maintheorem}

\section*{Overview}

In Section~\ref{sec:Ruelle-expanding} we recall the main properties of Ruelle-expanding maps, both as a guide
to the reader and for reference in the remaining of the text. The reader that is acquainted with the classical theory
of Ruelle expanding maps may choose to skip this section in a first reading of the text.

A sequence of Ruelle-Perron-Frobenius transfer operators suitable for the analysis of semigroup actions is defined in
Section~\ref{sec:RPFoperators}. In the case of random walks on the free semigroup, this sequence of operators coincides
with the iterates of a Ruelle-Perron-Frobenius transfer that averages the transfer operators of the dynamical systems that generate
the semigroup. Moreover, these operators are in a strong relation with the classical transfer operator for a locally constant
skew-product dynamics. To complement the classical approach of random dynamical systems, we are mainly interested in determining
intrinsic objects for the dynamics of the semigroup action. For that reason, we will, in particular, discuss the dependence of the
invariant measures for skew-product dynamics on the probability measures within the underlying shift that describes the
random walk on the free semigroup.

In Sections~\ref{sec:topol-entropy} and ~\ref{sec:zeta} we justify the notion of topological entropy for semigroup actions. Indeed,
using the relation with the skew product dynamics, we prove not only that the entropy can be computed by the growth rate of the
mean number of periodic orbits but also that it arises naturally as the radius of convergence of the zeta function (Theorems~\ref{thm:A}
and ~\ref{thm:B}).

Motivated by the strong connections found between the skew-product dynamics and the semigroup actions, in Sections~\ref{se.entropy} and \ref{sec:stationary-measure} we focus on building a bridge between the several concepts of the thermodynamic formalism for skew-products and the intrinsic objects for semigroup actions. In particular, we compare the notions of classical entropy, fibered entropy, relative
entropy, and quenched and annealed equilibrium states for symmetric and non-symmetric random walks. For instance, we prove that the entropy of the semigroup coincides with the fibered entropy of Ledrappier and Walters \cite{LW77} if and only if all maps have the same degree (Proposition~\ref{le:hfibered}); and that the later coincides with a quenched pressure in random dynamical systems (Proposition~\ref{le:hfibered-a}). The entropy of the semigroup is showed to coincide with an annealed pressure in random dynamical systems (Corollary~\ref{cor:Bufetov}) and also to be equal to the classical topological pressure
of a suitable potential for the skew-product dynamics (cf. \eqref{eq:equalpress}). Some results on stationary measures are
recalled as well while constructing absolutely continuous stationary measures in Theorem~\ref{thm:C}.

In Section~\ref{se.Therm-formalism} we select measures for the semigroup action using some variational insight. Using
both connections between the semigroup action, the classical thermodynamic formalism and the annealed pressure function, we
provide an intrinsic construction of measures (not necessarily stationary) on the ambient space which arise as
marginals from equilibrium states. The two different approaches allow us to conclude that such maximal entropy measures
can be computed using either an averaged equidistribution of periodic points or of preimages, which ends the proof of Theorem~\ref{thm:C}.

\section{Main properties of Ruelle-expanding maps}\label{sec:Ruelle-expanding}

Let $(X,d)$ be a compact metric space and $T:X \rightarrow X$ a Ruelle-expanding map.

\subsubsection*{\emph{\textbf{Expansivity}}}\label{expansive}

Any $\varepsilon$ verifying $\varepsilon < \min \, \{r, \frac{c}{1+\rho}\}$ is a constant of expansivity of $T$. This means that
$$d(T^n(x),T^n(y) \leq \varepsilon \quad \forall\,n\geq 0 \quad \Rightarrow \quad x=y.$$
This property ensures that, for any $n \in \mathbb{N}$, the periodic points with period $n$ are isolated and therefore in finite number. Yet, in this context, we also have
$$X=\bigcup_{n\geq 0}\,T^{-n}(\,\overline{\text{Per}(T)}\,)$$
where $\text{Per}(T)$ denotes the set of periodic points of $T$ (details in \cite{CM}).

\subsubsection*{\emph{\textbf{Pre-images}}}\label{preimages}

The cardinal of the pre-images by $T$ is uniformly bounded (that is, there is $p\in\NN$ such that $card(T^{-1}(\{x\}))\leq p$, for all $x\in X$) and locally constant. Moreover, if $\delta=\min\{r, \frac{c}{2\rho}\}$, and $x,y \in X$ are such that $d(x,y)<\delta$, then, for any positive integer $n$, we may write
$$T^{-n}(x)=\{x_1, x_2, \cdots, x_{k_n}\} \quad \text{ and } \quad T^{-n}(y)=\{y_1, y_2, \cdots, y_{k_n}\}$$
satisfying $d(T^j(x_i),T^j(y_i))\leq \rho^{n-j} \,d(x,y)$ for all $0\leq j \leq n$ and $1 \leq i \leq k_n.$ (See \cite{CM} for details.)

\subsubsection*{\emph{\textbf{Contractive branches}}}\label{contractivebranch}

Let $S\subseteq X$. Given $n\in\NN$, we say that $\phi:S\rightarrow X$ is a {\it contractive branch} of $T^{-n}$ if
\begin{enumerate}
	\item $(T^n\circ \phi)(x)=x \,\,\, \forall x \in S$
	\item $d((T^j\circ \phi)(x),(T^j\circ \phi)(y))\leq \rho^{n-j}\,d(x,y)\,\,\,\forall \, x,y\in S\,\,\,\forall \, j\in\{0,1,\ldots,n\}.$
\end{enumerate}
It is straightforward (cf. \cite{CM}) to conclude that, given $x \in X$, $n\in\mathbb{N}$ and $a \in T^{-n}(x)$, there is a contractive branch $\phi:B_r(x) \rightarrow X$ of $T^{-n}$ such that $\phi(x)=a$. Moreover, if $\varepsilon$ is a constant of expansivity, $\varsigma \leq \varepsilon$ and $B_{n,\varsigma,a}=\{z \in X: d(T^j(z), T^j(a))< \varsigma \quad \forall \, 0 \leq j \leq n\}$, then $B_{n,\varsigma,a}=\phi(B_{\varsigma}(x)).$

\subsubsection*{\emph{\textbf{Spectral decomposition of the dynamics}}}\label{spectraldecomposition}

In \cite{Ru2}, it was proved that there exists a (unique) finite family
$\displaystyle \big(\Lambda_i^{(m)}\big)_{i \in \{1, \cdots, n_m\}; \, \, m \in \{1, \cdots, M\}}$
of compact disjoint subsets, called basic components,
 such that
\begin{itemize}
\item[(C1)] $T(\Lambda_i^{(m)})=\Lambda_{i+1}^{(m)}$ for all $i \in \{1, \cdots, n_m-1\}$ and $m \in \{1, \cdots, M\}$.
\item[(C2)] $T(\Lambda_{n_m}^{(m)})=\Lambda_1^{(m)}$ for all $m \in \{1, \cdots, M\}$.
\item[(C3)] $\bigcup_{i,m}\,\,\Lambda_i^{(m)}=\overline{\text{Per}(T)}$.
\item[(C4)] $T^{n_m}_{|_{\Lambda_i^{(m)}}}$ is Ruelle-expanding.
\item[(C5)] For any open non-empty subset $V$ of $\Lambda_i^{(m)}$ there is $N\in \mathbb{N}$ such that $(T^{n_m})^N(V)=\Lambda_i^{(m)}$.
\end{itemize}
Notice that, if $X$ is connected, then $X=\overline{\text{Per}(T)}$ and $X$ reduces to one basic component (where, by (C5), $T$ is topologically mixing). Following \cite{L}, we say that a map $T:X \rightarrow X$ is covering if for any non-empty open set $V$ there exists an iterate $N \in \mathbb{N}$ such that $T^N(V) = X$. From (C5), every Ruelle-expanding topologically mixing map is covering.

\subsubsection*{\emph{\textbf{Dynamical zeta function and entropy}}}\label{zeta function}

The zeta function of $T$ is a formal series that encodes the information regarding the number of periodic points of $T$, given by
$$z \in \mathbb{C} \,\mapsto\, \zeta_T(z)=\exp\left(\sum_{n=1}^\infty\frac{N_n(T)}{n}\, z^n\right)$$
where $N_n(T)$ is the number of periodic points with period $n$, for all $n \in \mathbb{N}$ (cf. \cite{PP} and references therein). In our context, $\zeta_T$ is a rational function, so it has a meromorphic continuation to the whole complex plane. The poles, zeros and residues of this extension provide additional topological invariants for $T$ and an insight into its orbit structure. In particular, the topological entropy of $T$ is equal to $-\log(\varrho)$, where $\varrho$ is the radius of convergence of $\zeta_T$. Additionally, if $T$ is topologically mixing, then
$$h_{\text{top}}(T)=\lim_{n \rightarrow +\infty}\,\,\frac{1}{n}\, \log N_n(T)$$
a limit equal to zero if and only if $X$ is finite. The reader may find more details in \cite{CM}.

\subsubsection*{\emph{\textbf{Thermodynamic formalism}}}\label{RuelleTheorem}

Given a positive constant $\theta \in \,]0,1[$, consider the space of H\"older-continuous maps with exponent $\theta$
$$\mathfrak{F}_{\theta}^+=\left\{w: X \rightarrow \mathbb{R}: \exists\, C_w > 0: |w(x)-w(y)|\leq C_w\,d(x,y)^\theta \,\,\,\,\,\forall \,x,y \in X\right\}$$
with the norm
$$\|w\|_{\theta}= \|w\|_{\infty} + \sup_{x \neq y,\,\, d(x,y)<\delta}\,\, \frac{|w(x)-w(y)|}{d(x,y)^{\theta}}$$
where $\delta=\min\{r, \frac{c}{2\rho}\}$. The space of real continuous functions defined in $X$, with the uniform norm $\|\,\|_0$, will be denoted by $C^0(X)$. For a real valued function $\varphi:X \rightarrow \mathbb{R}$, the associated \emph{Ruelle operator} $\mathfrak{L}_\varphi$ acts on a map $\psi:X \rightarrow \mathbb{R}$ as
$$x \in X \,\mapsto\, \mathfrak{L}_\varphi(\psi)(x) \,=\, \sum_{y \,\in\, T^{-1}(x)}\,\,\psi(y)\,e^{\varphi(y)}.$$
This operator is well defined (see the properties of the pre-images of $T$), linear and positive. By induction, we may verify that, for any $\psi$ and $x$,
$$\mathfrak{L}^n_\varphi(\psi)(x)\,= \sum_{y \,\in\, T^{-n}(x)}\,\,\psi(y)\,e^{S_n\,\varphi(y)}$$
where $S_n\,\varphi(y)=\varphi(y)+\cdots+ \varphi(T^{n-1}y).$ If $\varphi \in C^0(X)$, then $\mathfrak{L}_\varphi(\psi) \in C^0(X)$ and the operator $\mathfrak{L}_\varphi$ restricted to $C^0(X)$ is continuous with respect to the $C^0$-norm. Similarly, if $\varphi, \, \psi \in \mathfrak{F}_{\theta}^+$, then $\mathfrak{L}_\varphi(\psi) \in \mathfrak{F}_{\theta}^+$. The most significant ergodic properties of Ruelle-expanding maps yielded by the associated Ruelle operator are described by the following theorem.

\begin{theorem}\emph{\cite{Ru2}}\label{te.Ruelle}
Assume that $T:X \rightarrow X$ is a topologically mixing Ruelle-expanding map and $\varphi \in \mathfrak{F}_{\theta}^+$. Then:
\begin{enumerate}
\item There is a simple maximal positive eigenvalue $\lambda_{\varphi}$ of
$\mathfrak{L}_\varphi: C^0(X) \rightarrow C^0(X)$
with a corresponding strictly positive eigenfunction $H \in \mathfrak{F}_{\theta}^+$.
\item $H$ is the unique positive eigenfunction of $\mathfrak{L}_\varphi$, up to scalar multiplication.
\item The remainder of the spectrum of $\mathfrak{L}_{\varphi}:\mathfrak{F}_{\theta}^+ \rightarrow \mathfrak{F}_{\theta}^+$ is contained in a disk centered at $(0,0)$ with radius $R <\lambda_{\varphi}$.
\item There is a unique probability $\nu$ on the Borel subsets of $X$ such that $\mathfrak{L}_{\varphi}^*\nu=\lambda_{\varphi}\,\nu$ and $\int Hd\nu=1$.
\item For any $\psi \in C^0(X)$, $\lim_{n\rightarrow + \infty}\,\,\|\lambda_{\varphi}^{-n}\,\,\mathfrak{L}_{\varphi}^n (\psi) - H\int \psi\,d\nu\|_0=0.$
\item For any $\psi \in \mathfrak{F}_{\theta}^+$, $\sup_{n \,\in \,\mathbb{N}}\,\|\lambda_{\varphi}^{-n}\, \mathfrak{L}^n_\varphi(\psi)\|_{\theta} < \infty.$
\item The probability $\mu=H\nu$ is $T$-invariant, exact, positive on non-empty open sets and satisfies a variational principle: it is the unique probability such that
$$\log\,(\lambda_{\varphi})=h_{\mu}(T) + \int \varphi\,d\mu= \max_{u \,\in \,\mathfrak{M}(X,T)}\,\,\displaystyle \left(h_{u}(T) + \int \varphi\,d\,u\right)$$
where $\mathfrak{M}(X,T)$ denotes the (compact, convex) space of $T$-invariant probability measures on the $\sigma$-algebra of the Borel subsets of $X$.
\end{enumerate}
\end{theorem}

\begin{remark}\label{re.particular_cases}
In the particular case of $T:X \to X$ being a $C^2$ expanding map of a compact connected Riemannian manifold $X$ endowed with a Lebesgue probability measure ($\Leb$ for short), we may consider $\varphi:X \to \mathbb{R}$ given by $\varphi(x)=-\log\,|det\,DT_x|$, which belongs to $\mathfrak{F}_{\theta}^+$. Then 
there is a unique probability measure $\mu$ on the Borel sets of $X$ which is invariant under $T$ and absolutely continuous with respect to Lebesgue. Moreover, the density of $\mu$ with respect to Lebesgue is precisely given by $H$, so it is H\"older and strictly positive. Additionally, $\mu$ is exact; $h_\mu(T) = \int\, \log\,|det\,DT_x|\,d\mu$; $h_u(T) < \int\, \log\,|det\,DT_x|\,d\,u$, for every $T$-invariant probability measure $u \neq \mu$; and, for any Borel set $A$, we have $\lim_{n \to +\infty}\,\Leb(T^{-n}(A))=\mu(A)$.

If, instead, we consider $\varphi \equiv 0$ and $X$ is connected, then we may deduce relevant information concerning the asymptotic distribution of the pre-images of each point by the expanding map $T:X \to X$. Given $x \in X$ and $n \in \mathbb{N}$, let $\nu_n(x)$ be the probability measure on the Borel subsets of $X$ defined as
\begin{equation}\label{eq:convmeasures}
\nu_n(x)=\frac{1}{\deg\,(T)^n}\,\sum_{T^n(y)=x}\,\delta_y
\end{equation}
where $\delta_y$ is the Dirac measure supported on $\{y\}$ and $\deg\,(T)$ is the degree of $T$, that is, $\deg\,(T)=\sharp\,T^{-1}(\{a\})$, a number independent of $a \in X$. Then 
there exists a unique probability measure $\nu$ on the Borel sets of $X$ such that, for every $x \in X$,
$$\nu=\lim_{n \to +\infty}\,\nu_n(x)$$
in the weak$^{\ast}$ topology. Such a $\nu$ is invariant under $T$ (that is, $H \equiv 1$), exact, positive on non-empty open sets, its entropy is equal to
$h_\nu(T) = \log\,(\deg\,(T))$
(that is, $\lambda_\psi=\deg\,(T)$) and it maximizes entropy (that is, $h_\nu(T) > h_v(T)$ for every $T$-invariant probability measure $v \neq \nu$). Moreover, $\nu=\mu$ if and only if, for all $x \in X$ such that $T^n(x)=x$, we have $|\det\,D_x(T^n)|=\deg\,(T)^n.$ We refer the reader to \cite{Bo75} for more information.
\end{remark}

\subsubsection*{\emph{\textbf{Gibbs measures}}}\label{sec.Gibbs}

Each equilibrium state $\mu$ obtained in Ruelle's Theorem has a positive Jacobian, namely $J_\mu f= e^{-\varphi}$, so it is positive on open non-empty sets. Besides, it is a Gibbs measure, that is, given a finite partition $\mathcal{P}=\{P_1, \cdots, P_{\ell}\}$ of $X$ with diameter less than $r$, there exists a constant $C>1$ such that, for each positive integer $m \geq 0$, any $1\leq i \leq \ell$, every contractive branch $\phi:P_i\rightarrow X$ of $T^{m}$ and all $x \in \phi(P_i)$, we have
$$ C^{-1} \leq \frac{\mu(\phi(P_i))}{e^{-m \, \log(\lambda_{\varphi})+S_m\,\varphi(x)}}\leq C.$$

\subsubsection*{\emph{\textbf{Examples}}}\label{ex.unilateral-shift} One-sided Markov subshifts of finite type, determined by aperiodic square matrices with entries in $\{0,1\}$, are Ruelle-expanding map (with $r=1$ and $c=\rho=1/2$) and topologically mixing \cite{Wa}. 
If $M$ is a compact Riemannian manifold without boundary and $T:M\rightarrow M$ a $\mathcal{C}^1$ map, the dynamical system $T$ is said to be $C^1$-expanding if
$$\exists\,\,\lambda > 1:\,\,\forall\, x\in M\,\,\forall\, v \in T_x M\,\,\left\|D_x T(v)\right\|\geq \lambda \left\|v\right\|.$$
It is easy to prove that, in the $C^1$ context, $T$ is $C^1$-expanding if and only if it is Ruelle-expanding.
For instance, $T:\mathbb{S}^1 \rightarrow \mathbb{S}^1$, $T(z)=z^m$, is $C^1$-expanding for all positive integer $m>1$ (in this case $\lambda=m$). Moreover, if $T:\mathbb{S}^1\rightarrow \mathbb{S}^1$ is a $C^2$ map with degree bigger than one such that $DT(z)\neq 0$ for all $z \in \mathbb{S}^1$ and all the periodic points of $T$ are hyperbolic (a generic property), then the restriction of $T$ to the complement of the union of the basins of the sinks is Ruelle-expanding. More generally, if $L:\mathbb{R}^n \rightarrow \mathbb{R}^n$ is a linear map whose eigenvalues have absolute value bigger than one and such that $L(\mathbb{Z}^n)\subseteq \mathbb{Z}^n$, then $L$ induces a Ruelle-expanding map on the flat torus $\mathbb{R}^n / \mathbb{Z}^n$. Conversely, any $C^1$-expanding map in the $n$-dimensional flat torus is topologically conjugate to one obtained by this process; see \cite{Shu3}.

\section{Ruelle-Perron-Fr\"obenius operators}\label{sec:RPFoperators}

In this section we shall present a proposal for the notion of Ruelle-Perron-Fr\"obenius transfer operator to be assigned to a semigroup action which is a natural extension of the concept of transfer operator for an individual dynamical system (cf. Remark~\ref{rmk:equality}). The operator we will introduce depends on the chosen set $G_1$ of generators of $G$ and on the selected random walk $R_{\eta}$ on $G$; we will come back to this subject on Subsection~\ref{symmetry}.

Let $G$ be a semigroup generated by a finite subset $G_1$ of Ruelle-expanding maps acting on a compact connected metric space $M$. Consider the corresponding continuous semigroup action $S: G \times M \to M$. Given a continuous observable $\varphi: M \to \mathbb R$, let $\mathfrak{L}_{\underline g, \varphi} : C^0(M) \to C^0(M)$ denote the usual Ruelle-Perron-Fr\"obenius operator 
associated to the dynamical system $\underline{g}$ and the observable $\varphi$:
\begin{equation}
\mathfrak{L}_{\underline g, \varphi} \,(\psi) (x)
	= \sum_{\underline g(y)=x} e^{\varphi(y)} \, \psi(y).
\end{equation}
Therefore, for each $k \in \mathbb{N}$,
$$\mathfrak{L}^k_{\underline g, \varphi}\,(\psi) (x)
	= \sum_{{\underline g}^k(y)=x} e^{S_{k}\,\varphi(y)} \, \psi(y),
$$
where $S_{k} \,\varphi(x)=\sum_{\ell=0}^{k-1} \varphi(\underline{g}^\ell(y))$
and $\underline{g}^\ell := (g_{i_n} \dots g_{i_1})^\ell$ for every $0\le \ell \le k-1.$

Alternatively, for any $\underline g=g_{i_n}...g_{i_1} \in G_n$, the operator $\mathcal{L}_{\underline g, \varphi} : C^0(M) \to C^0(M)$ is defined as
\begin{equation}\label{RPF_sumsg}
\mathcal{L}_{\underline g, \varphi}\,(\psi) (x)
	= \mathfrak{L}_{\underline g, \varphi_{\underline g}}\,(\psi)  (x)
	= \sum_{\underline g(y)=x} e^{\varphi_{\underline g}(y)} \, \psi(y),
\end{equation}
where $\varphi_{\underline g}(y)=\sum_{m=0}^{n-1} \varphi(\underline{g}_m(y))$, $g_0=id$ and $\underline{g}_m = g_{i_m} \dots g_{i_1}$ for every $1 \le m \le n-1$. Observe that, if $\varphi$ is continuous (respectively, H\"older) and the elements of $G_1$ are $C^2$  expanding maps, then $\varphi_{\underline g}$ is continuous (respectively, H\"older) as well. Moreover, for each $k \in \mathbb{N}$,
\begin{equation*}\label{RPF_n1}
\mathcal{L}^k_{\underline g, \varphi}\,(\psi) (x)
	= \sum_{{\underline g}^k(y)=x} e^{S_{k}\,\varphi_{\underline g}(y)} \, \psi(y)
\end{equation*}
where, as usual, for every $\phi:M \to \mathbb{R}$, we write $S_{k} \,\phi(x)=\sum_{\ell=0}^{k-1} \phi(\underline{g}^\ell(y))$ and $\underline{g}^\ell = (g_{i_n} \dots g_{i_1})^\ell$ for every $0\le \ell \le k-1$. Roughly speaking, the transfer operator $\mathcal{L}_{\underline g, \varphi}$ gathers the information of the pathwise transfer operators while we evaluate the observable along the $n$th iteration of the semigroup dynamics, as specified by \eqref{RPF_sumsg}.
It is not hard to check that
$
\mathcal{L}_{\underline g,\varphi}
	= \mathfrak{L}_{g_{i_n},\varphi} \circ \mathfrak{L}_{g_{i_{n-1}}, \varphi} \dots \circ \mathfrak{L}_{g_{i_1},\varphi}$
for any $\underline g=g_{i_n}...g_{i_1} \in G_n$ and every positive integer $n$.

\begin{remark} The operator $\mathcal{L}_{\underline g, \varphi}$ depends on the word that expresses $\underline{g}$ as a combination of elements of the generator $G_1$ because it is based on the function
$\varphi_{\underline g}(y)=\sum_{m=0}^{n-1} \varphi(\underline{g}_m(y))$ which may change if the order of the concatenation is altered. Thus, in this definition, we are distinguishing different concatenations even if they correspond to the same endomorphism of $M$.
\end{remark}

\subsubsection*{\emph{\textbf{A sequence of transfer operators}}}\label{RPF3}
One may also consider the non-stationary dynamical system whose complexity is indexed by the `time $n$', corresponding to the `ball of radius $n$' in the semigroup. Such viewpoint has turned to be very fruitful in the description of the topological entropy and the complexity of group and semigroup actions \cite{RoVa1}, and motivates the definition of the following weighted mean sequence of transfer operators.

\begin{definition}
\label{def:PFsequence}
Given a continuous potential $\varphi: M \to \mathbb R$ and a continuous finitely generated semigroup action $S:G \times M \to M$, the \emph{Ruelle-Perron-Fr\"obenius sequence} of $G$ with respect to $\varphi$ is the sequence $(\mathbf{L}_{n,\varphi})_{n\ge 1}$ of bounded linear operators acting on $C^0(M)$ and given, for every $n\ge 1$, by
\begin{equation*}
\mathbf{L}_{n,\varphi}
	=\frac{1}{p^n}\sum_{|\underline{g}|=n} \mathcal{L}_{\underline g,\varphi} = \frac{1}{p^n}\sum_{|\underline g|=n} \mathfrak{L}_{\underline g,\,\varphi_{\underline g}}.
\end{equation*}
\end{definition}
We observe that, for each potential $\varphi$, this sequence is obtained by averaging the usual Ruelle-Perron-Fr\"obenius transfer operators associated to $\varphi_{\underline g}$ of each dynamics $\underline g$ in $G_n^*$.
Furthermore, notice that the operator $\mathcal{L}_{\underline g, \varphi}$ depends on the order of the concatenation of the generators that build $\underline{g}$. So, we are averaging not on the maps in $G_n$ but on the different words that express them in terms of the generators.

\begin{example}\label{rmk:equality}
If $G$ is generated by $G_1=\{id,  f \}$, then $G_n^* = \{f^n\}$ for $n \geq 1$
and
$
\mathbf{L}_{n,\varphi}
	=  \mathcal{L}_{f^n,\varphi}
	=  \mathfrak{L}_{f^n, \varphi_{f^n}}
	= \mathfrak{L}^n_{f,\varphi}.
$
In particular, $\mathbf{L}_{1,\varphi}= \mathcal{L}_{f,\varphi}= \mathfrak{L}_{f,\varphi}$ and these three maps coincide with the usual Ruelle-Perron-Fr\"obenius operator for $f$. Therefore, they are all natural extensions of the notion of transfer operator for an individual dynamical system. Moreover, if $\varphi\equiv 0$,
$\mathcal{L}_{\underline g, 0}\, (\psi) (x) =\sum_{\underline g(y)=x}\, \psi(y)$ and
corresponds to the Ruelle-Perron-Fr\"obenius transfer operator $\mathfrak{L}_{\underline g, 0}$ associated to the topological entropy of $\underline g:M \to M$; see Subsection~\ref{RuelleTheorem}. For the same observable $\varphi\equiv 0$, the sequence of transfer operators $\mathbf{L}_{n,0}$ is given by
$$\mathbf{L}_{n,0}\,(\psi)(x)
	=\frac{1}{p^n}\sum_{|\underline{g}|=n}\,\sum_{\underline{g}(y)=x} \,\psi(y)$$
for every $\psi \in C^0(M)$ and $x \in M$.
\end{example}

\subsection*{Averaged and fibered transfer operators for the semigroup action}\label{sec:averagedRPF2}

Given a finitely generated continuous semigroup action $G \times M \to M$,  a shift-invariant probability measure $\eta$ on $\Sigma_p^+$ and a continuous potential $\varphi: M \to \mathbb{R}$, take the non-stationary sequence $(\mathbf{L}_{n,\eta,\varphi})_{n\ge 1}$ given by
\begin{align*}\label{Lnonstat}
n\ge1 \mapsto \mathbf{L}_{n,\eta, \varphi}
	= \int_{\Sigma_p^+}\, \mathcal{L}_{g_{i_n}, \varphi} \dots
	\mathcal{L}_{g_{i_2}, \varphi} \mathcal{L}_{g_{i_1}, \varphi} \,
	(\textbf{1})\,d \eta([i_1, \dots, i_n]).
\end{align*}
In the case $\eta=\eta_{\underline p}$, which describes the symmetric random walk, the family $(\mathbf{L}_{n,\eta_{\underline p},\varphi})_{n\ge 1}$ coincides with the family of transfer operators introduced in Definition~\ref{def:PFsequence}.
In the case of a random walk $\eta=\eta_{\underline a}$ associated to a non-trivial probability vector
$\underline a$, the a priori non-stationary sequence becomes stationary. Indeed, given $\underline \varphi
=(\varphi_1, \dots, \varphi_p) \in C^0(M)^p$ and a non-trivial probability vector $\underline a$,
the transfer operator $\tilde{\mathbf{L}}_{\underline a, \underline \varphi}$ acting on $C^0(M)$
is precisely
\begin{equation}\label{eq:RPF-integratedM}
\phi \in C^0(M) \quad \mapsto \quad  \tilde{\mathbf{L}}_{\underline a, \underline \varphi} \phi(x)
	= \sum_{i=1}^p   a_i  \sum_{g_i(y)=x} e^{\varphi_i(y)} \phi(y).
\end{equation}

\subsection*{Averaged and fibered transfer operators for the skew-product}\label{sec:averagedRPF1}

The semigroup actions is naturally associated to a skew-product dynamics defined by
\begin{equation*}
\begin{array}{rccc}
\mathcal{F}_G : & \Sigma_p^+  \times M & \to & \Sigma_p^+  \times M \\
	& (\omega,x) & \mapsto & (\sigma(\omega), g_{\omega_1}(x))
\end{array}
\end{equation*}
where $\omega=(\omega_1,\omega_2, \dots)$. Given a non-trivial probability vector $\underline a=(a_1, a_2, \ldots,a_p)$, where $a_k>0$ for all $k = 1, \dots, p$ and $\sum_{k=1}^{p}\,a_k = 1$, consider the Bernoulli probability measure $\eta_{\underline a} = {\underline a}^{\mathbb N}$
on $\Sigma_p^+$. 
Inspired by the work \cite{Baladi} on random expanding maps, we assign to any $\underline \varphi = (\varphi_1, \dots, \varphi_p) \in C^0(\Sigma_p^+ \times M)^p$ the integrated transfer operators
$$\mathbf{\hat L}_{\underline a, \underline \varphi} : C^0(\Sigma_p^+ \times M) \to C^0(\Sigma_p^+ \times M)$$
defined by
\begin{equation}\label{eq:RPF-integratedSigM}
\mathbf{\hat L}_{\underline a, \underline \varphi}  \psi(\omega,x)
	= \sum_{i=1}^p   a_i \; \mathcal{L}_{g_i,\varphi_i} \psi( i\omega,x)
	= \sum_{i=1}^p   a_i  \sum_{g_i(y)=x} e^{\varphi_i(i\omega, y)} \psi( i\omega,y)
\end{equation}
for every $\psi \in C^0(\Sigma_p^+ \times M)$, where $i\omega$ stands for the sequence $(i, \omega_1, \omega_2, \dots)$.
Observe that, in the particular case of $\underline \varphi=(0, \dots, 0)$, one gets
$\mathbf{\hat L}_{\underline a, \underline \varphi} 1 = \int \deg \,(g_i) \; d\underline a(i)$.
Moreover, $\tilde{\mathbf{L}}_{\underline a, \underline \varphi}^n = \mathbf{L}_{n,\eta_{\underline a}, \varphi}$ for all $n\ge 1.$ Furthermore, any $\varphi\in C^0(M)$ induces a sequence $\underline\varphi(\omega,x) =( \varphi(x), \dots, \varphi(x) ) \in C^0(\Sigma_p^+ \times M)^p$; and, given $\psi \in C^0(\Sigma_p^+ \times M)$ that does not depend on $\omega$, if we consider
$\phi:M \to \mathbb{R}$ defined by $\phi(x)=\psi(1,x)$, then we have
\begin{equation}\label{eq:RPF-integrated}
 \mathbf{\hat L}_{\underline a, \underline \varphi} \, \psi= \mathbf{\tilde L}_{\underline a, \underline \varphi} \,\phi.
\end{equation}
As
$\mathbf{\hat L}_{\underline a, \underline \varphi}$ and $\mathbf{\tilde L}_{\underline a, \underline \varphi}$
are positive operators, their spectral radius are equal to the exponential growth rates of
$\|\mathbf{\hat L}^n_{\underline a, \underline \varphi} 1_{\Sigma_p^+\times M}\|_0$
and $\|\mathbf{\tilde L}^n_{\underline a, \underline \varphi} 1_M\|_0$, respectively. Thus,
$sp(\mathbf{\hat L}_{\underline a, \underline \varphi}) = sp(\mathbf{\tilde L}_{\underline a, \underline \varphi})$.

\begin{remark}
A similar link between transfer operators on the phase space and on the skew-product dynamics has been considered previously in \cite{S00}. However, in this reference the operators are built averaging \emph{normalized} transfer operators for individual dynamics.
\end{remark}

\section{Topological entropy of the semigroup action}\label{sec:topol-entropy}

This section is devoted to the proof of Theorem~\ref{thm:A}, which relates the spectral radius of the Ruelle-Perron-Fr\"obenius sequence of $G$ (Definition~\ref{def:PFsequence}) with the exponential growth rate of periodic points and the topological entropy. First, let us recall the concept of separated points and topological entropy of a semigroup action adopted in \cite{Bufetov,RoVa1}.
Given $\varepsilon>0$ and $\underline g :=g_{i_{n}} \dots g_{i_2} \, g_{i_1}\in G_n$, the \emph{dynamical ball} $B(x,\underline g,\varepsilon)$ is the set
\begin{align}\label{eq:dynball}
B(x,\underline g,\varepsilon)\nonumber
	&:= \Big\{y\in X: d( \underline g_j (y), \underline g_j (x) ) \le \varepsilon, \; \text{ for every } 0\le j \le n \Big\}
\end{align}
where, as before,  for every $1 \le j \le n-1$ we denote by $\underline g_j$ the concatenation $g_{i_{j}} \dots g_{i_2} \, g_{i_1}\in G_j$, and $\underline g_0=id$. We also assign a dynamical metric $d_{\underline g}$ to $M$ by setting
 \begin{equation}\label{eq:dg}
  d_{\underline g} (x, y)
 	:=   \max_{0\le j \le n } \, d(\underline{g}_j(x),\underline{g}_j(y)).
 \end{equation}
It is important to notice that both the dynamical ball and the metric depend on the underlying concatenation of generators $g_{i_n} \dots g_{i_1}$ and not on the group element $\underline g$, since the latter may have distinct representations.

Given $\underline g= g_{i_n} \dots g_{i_1} \in G_n$, we say that a set $K \subset M$ is \emph{$(\underline g, n, \varepsilon)$-separated} if $d_{\underline g}(x,y) > \varepsilon$ for any distinct $x,y \in K$.
The maximal cardinality of a  $(\underline g,\varepsilon, n)$-separated set on $M$ will be denoted by $s(\underline g, n, \varepsilon)$.
The topological entropy of a semigroup action estimates the growth rate in $n$ of the number of orbits of
length $n$ 
up to some small error $\varepsilon$.

\begin{definition}\label{de.top_entropy}
The \emph{topological entropy} of the semigroup action $S:G\times M \to M$ is the limit
\begin{align}
h_{\text{top}}(S)
    =\lim_{\varepsilon\to 0}\,\,\limsup_{n\to\infty}\frac1n\log\Big(\frac1{p^n}\sum_{|\underline g|=n}\,s(\underline g,n,\varepsilon)\Big).
\end{align}
\end{definition}

\begin{remark} This notion of topological entropy should be compared with the one proposed by Ghys, Langevin and Walczak, where the authors compute the asymptotic exponential growth rate of points that are separated by some group element (see \cite{GLW} for the precise definition). This corresponds to the largest exponential growth rate, while the definition we adopt here observes the growth rates of separated points averaged along semigroup elements.
\end{remark}

In the context of Ruelle-expanding dynamics this notion is connected to the largest exponential growth rate of periodic points (cf. Subsection~\ref{zeta function}). This is why we also consider the following asymptotic speed.

\begin{definition}\label{de.per_entropy}
The \emph{periodic entropy} of the semigroup action $S: G \times M \to M$ is the limit
\begin{align}\label{def:per growth}
\wp(S)=\limsup_{n\to + \infty} \,\frac{1}{n}\log \, (\max\{ N_n(G),1 \} )
\end{align}
where
$$N_n(G)=\frac{1}{p^n}\sum_{|\underline g|=n}\, \sharp \,\text{Fix}(\underline g).$$
\end{definition}

Observe that in order for $\wp(S)$  be a finite value, the set $\text{Fix}(\underline g)$ must be finite
for each $\underline g \in G\setminus\{id\}$, which holds for instance when $\underline{g}$ is expansive.
A map $\underline g \in G$ is said to be \emph{expansive} if there exists $\varepsilon_{\underline g} > 0$ such that, whenever $x,y \in M$ and $x\neq y$, then
$$\max\,\{d(\underline{g}^\ell(x),\underline{g}^\ell(y)) \colon \ell \in \mathbb{N}_0\} \geq \varepsilon_{\underline{g}}.$$

Within Ruelle-expanding dynamics, both notions of entropy (topological or periodic) may be estimated from the knowledge of the spectral radius of the  Ruelle-Perron-Fr\"obenius operator. The similar concept for semigroups is as follows.

\begin{definition}
Given $\varphi\in C^0(M)$, the \emph{spectral radius} $sp\,((\mathbf{L}_{n,\varphi})_{n \in \mathbb{N}})$ of the Ruelle-Perron-Fr\"obenius sequence $(\mathbf{L}_{n, \varphi})_{n \in \mathbb{N}}$ of positive operators in $C^0(M)$ is defined as
$$\log \,sp\,((\mathbf{L}_{n, \varphi})_{n})
	:= \limsup_{n\to\infty} \frac1n \log \|\mathbf{L}_{n,\varphi} \,(\textbf{1})\|_{C^0}.$$
\end{definition}
Notice that, if $\varphi\equiv 0$ and $G_1 \subset End^1(M)$, then
\begin{equation}\label{eqL0}
\mathbf{L}_{n,0} \,(\textbf{1})(x)
	= \frac1{p^n} \sum_{|\underline g|=n} \mathcal{L}_{\underline g,0} \,(\textbf{1})(x)
	= \frac1{p^n} \sum_{|\underline g|=n} \sharp \,\underline g^{-1}(x)
	= \frac1{p^n} \sum_{|\underline g|=n} \deg\,(\underline g)
\end{equation}

After \cite{RoVa1}, we know that an action of a finitely generated semigroup of $C^1$-expanding maps is strongly $\delta$-expansive, for some $\delta>0$, a notion that we now recall.

\begin{definition}\label{de.delta_expansive}
Given $\delta>0$, we say that a continuous semigroup action $S:G\times M \to M$ on a compact Riemannian manifold $M$ is \emph{$\delta$-expansive} if, whenever $x \not= y\in M$, there exist $\kappa\in \mathbb N$ and $\underline g\in G_\kappa$ such that
$d(\underline g(x),\underline g(y))>\delta.$
$S$ is said to be \emph{strongly $\delta$-expansive} if, for any $\gamma>0$, there exists $\kappa_\gamma\geq 1$ such that, for every $x\not= y\in M$ with $d(x,y)\geq\gamma$, for all $\kappa\geq \kappa_\gamma$ and any $\underline g\in G^*_\kappa$, we have
$
d_{\underline g}(x,y) =  \max_{0\le j \le n } \, d(\underline{g}_j(x),\underline{g}_j(y)) >\delta.
$
\end{definition}

This is a key property that eases our task of computing the topological entropy of a semigroup action. Indeed, when the action is strongly $\delta$-expansive, the topological entropy can be computed independently of $\varepsilon$. More precisely,

\begin{lemma}\cite[Theorem~25]{RoVa1}\label{le.varepsilon}
Let $G$ be the semigroup generated by a set $G_1=\{Id,g_1,\dots,g_p\}$, where $G_1^*$ is a finite set of Ruelle-expanding maps on a compact metric space $M$ and $S:G\times M \to M$ its continuous semigroup action. Take $0<\varepsilon<\delta$. Then
$$h_{\text{top}}(S)=\limsup_{n\to\infty}\frac1n\log\Big(\frac1{p^n}\sum_{|\underline g|=n}\,s(\underline g,n,\varepsilon)\Big).$$
\end{lemma}

Additionally, it was proved in \cite{RoVa1} that the topological entropy is a lower bound for the exponential growth rate of periodic points. For that purpose, the authors introduced the following pathwise specification property.

\begin{definition}\label{def:orbital-sepc}
We say that the continuous semigroup action $S: G\times M \to M$ associated to the finitely generated semigroup $G$ satisfies the \emph{(strong) orbital specification property} if, for any $\vep>0$, there exists $T(\vep)>0$ such that, given $k \in \mathbb{N}$, for any $\underline h_{p_j}\in G^*_{p_j}$ with $p_j \ge T(\vep)$ for every $1\leq j\leq k$, for each choice of $k$ points $x_1, \dots, x_k$ in $M$, for any natural numbers $n_1, \dots, n_k$ and any semigroup element
$\underline g_{n_j, j} = g_{i_{n_j}, j} \dots g_{i_2,j} \, g_{i_1,j} \in G_{n_j},$
where $j \in \{1,\dots, k\}$, there exists $x\in M$ such that
$
d(\underline g_{{\ell}, 1} (x),\,\, \underline g_{{\ell}, 1} (x_1) ) < \vep \; \text{for all} \,  1 \le \ell \le n_1
$
and
$
d(\underline g_{{\ell}, j} \, \underline  h_{p_{j-1}} \, ... \, \underline g_{{n_2}, 2} \, \underline h_{p_1} \, \underline g_{{n_1}, 1} (x), \,\, \underline g_{{\ell}, j} (x_j)) <  \vep
$
for all $2 \le j \le k, \,\, 1 \le \ell \le n_j$ where $\underline g_{\ell, j} = g_{i_{\ell}, j} \dots g_{i_1,j}$. We say that the semigroup action satisfies the \emph{periodic orbital specification property} if the point $x$ can be chosen periodic.
\end{definition}

\begin{theorem}\cite[Theorem~28]{RoVa1}\label{thm:RoVaI}
Let $G$ be the semigroup generated by $G_1=\{Id,g_1,\dots,g_p\}$, where $G_1^*$ is a finite set of Ruelle-expanding maps on a compact connected metric space $M$ and $S:G\times M \to M$ its continuous semigroup action. Then $G$ satisfies the periodic orbital specification property and
$$ 0<h_{\text{top}}(S) \leq \limsup_{n\to\infty} \frac{1}{n}\log \Big( \frac{1}{p^n} \sum_{|\underline g|=n} \,\sharp \,\text{Fix}(\underline g) \Big).$$
\end{theorem}

We will show that the equality holds and that the previous $\limsup$ is indeed a limit.

\subsection{Proof of Theorem~\ref{thm:A}}

Let $G$ be the semigroup generated by $G_1=\{Id,g_1,\dots,g_p\}$, with $G_1^*$ a finite set of Ruelle-expanding maps on a compact connected metric space $M$. From \cite[Theorem~28]{RoVa1}, we already know that $h_{\text{top}}(S)\leq \wp(S)$. We are left to show the opposite inequality. First we will prove that the class of Ruelle-expanding maps is closed under concatenation, and so forms a semigroup.

\begin{lemma}\label{le:R-expanding}
If each map in the finite  set
$G_1^*$ is Ruelle-expanding, then $\underline{g}$ is Ruelle-expanding for
any $\underline{g} \in G-\{Id\}$. Moreover, there exists $\delta>0$ such that the semigroup action $S: G \times M \to M$
is strongly $\delta$-expansive.
\end{lemma}

\begin{proof}
Assume that $g_1$ and $g_2$ are Ruelle-expanding maps. We claim that the composition $g_2\,g_1$ is a Ruelle-expanding map.
Let $c_1>0$ and $c_2>0$ be such that, for all $x, y \in M$ with $x\neq y$
$$
g_1(x)=g_1(y) \, \Rightarrow \, d(x,y)>c_1
	\quad\text{and}\quad
	g_2(x)=g_2(y) \, \Rightarrow \, d(x,y)>c_2.
$$
Consider the composition $g_2\,g_1$ and $x, y \in X$ with $x\neq y$ and assume that $g_2\,g_1(x)=g_2\,g_1(y)$. Then,
either $g_1(x)=g_1(y)$, in which case $d(x,y)>c_1$; or $g_1(x)\neq g_1(y)$, in which case one has $d(g_1(x),g_1(y))>c_2$.
In the latter case, observe that since $g_1$ is uniformly continuous, there is $\delta_1>0$ such that $d(g_1(a),g_1(b))\leq c_2$ whenever
 $d(a,b) \leq \delta_1$. Thus, from $d(g_1(x),g_1(y))>c_2$ we conclude that $d(x,y)> \delta_1$.
Therefore, if $c_{12}=\min\,\{c_1, \delta_1\}>0$, then
$$
\forall \,x, y \in M:\,\, x\neq y,\,\, g_2\,g_1(x)=g_2\,g_1(y) \, \Rightarrow \, d(x,y)>c_{12}.
$$
Now, let $r_1$, $r_2$, $\rho_1 <1$ and $\rho_2 <1$ be positive constants such that, for each $x \in M$, for every $a \in g_2^{-1}(\{x\})$ and every $b \in g_1^{-1}(\{a\})$, there exists a map $\phi_2: B_{r_2}(x) \rightarrow M$, defined on the open ball centered at $x$ with radius $r_2$ such that $\phi_2(x)=a$, $g_2\circ \phi_2 = id$ and
$$d(\phi_2(z),\phi_2(w))\leq \rho_2 \,d(z,w) \quad \forall\,z,w \in B_{r_2}(x),$$
so $\phi_2(B_{r_2}(x))\subset B_{\rho_2\,r_2}(a)$; and there is another map $\phi_1: B_{r_1}(a) \rightarrow M$, defined on the open ball centered at $a$ with radius $r_1$, satisfying $\phi_1 (a)=b$, $g_1\circ \phi_1 = id$ and
$$d(\phi_1(z),\phi_1(w))\leq \rho_1 \,d(z,w) \quad \forall\,z,w \in B_{r_1}(a).$$
The uniform contraction rate $\rho_2$ of $\phi_2$ and the uniform size of the balls associated to the contractive branches of $g_1$ allow us to find
$
0 <r_{12} \leq  r_2
$
such that, for any $x \in M$ and every $a \in g_2^{-1}(\{x\})$, we have
$
\phi_2\,(B_{r_{12}}(x)) \subset B_{r_1}(a).
$
For instance, we may take $r_{12} = \min\,\{r_1, r_2\}$. Therefore, the map $\phi_{12}: B_{r_{12}}(x) \rightarrow M$
given by
$\phi_{12} \equiv \phi_1\,\phi_2$
is well defined and is a contractive branch for $g_2\,g_1$, since it has the following properties:
\begin{enumerate}
\item $\phi_{12} (x)= \phi_{1}\,\phi_{2}(x)=\phi_{1}(a)= b;$
\item $(g_2\,g_1)\circ \phi_{12} = (g_2\,g_1)\circ (\phi_{1}\, \phi_{2}) = g_2 \, (g_1\circ \phi_{1}) \, \phi_{2} = id$;
\item $d(\phi_{12}(z),\phi_{12}(w))\leq \rho_1 \,d(\phi_{2}(z),\phi_{2}(w)) \leq \rho_1\,\rho_2 \,d(z,w)$, for all $z,\,w \in B_{r_{12}}(x).$
\end{enumerate}
Thus, $g_2\,g_1$ is Ruelle-expanding, which proves our claim.

The previous computations also yield that, if $\rho_i\in (0,1)$ denotes the backward contraction rate for $g_i \in G_1^*$ and $r_i>0$ is so that all inverse branches for $g_i$
are defined in balls of radius $r_i$, then every map $g_{i_2} g_{i_1}$ is Ruelle-expanding and its inverse branches are defined in balls of radius $r$ with backward contraction rates $\rho$, where
\begin{equation}\label{eq:constants-Ru}
r=\min\{ r_i : 1\le i \le p\}
	\quad\text{and}\quad
	\rho =\min\{ \rho_i : 1\le i \le p\}.
\end{equation}

We now proceed by induction on $n$. If, for a fixed positive integer $n$, the concatenation of $n$ Ruelle-expanding maps is Ruelle-expanding, then, considering $n+1$ such maps, say
$g_{i_{n+1}} \,g_{i_{n}} \,\cdots \,g_{i_1},$
we may split their composition into the concatenation of two Ruelle-expanding maps
$g_{i_{n+1}} \, (g_{i_{n}} \,\cdots \,g_{i_1})$
and apply what we have just proved. This finishes the proof of the first part of the lemma.

We now prove that $S : G \times M \to M$ is strongly $\delta$-expansive for some $\delta>0$. Set $\delta=\frac{r}2$.
Given $\gamma>0$, let $\kappa_\gamma\ge 1$ be such that $\rho^{\kappa_\gamma} \delta < \gamma$, where $\rho$ is defined by ~\eqref{eq:constants-Ru}.
Now, given any points $x\not= y\in X$ with $d(x,y)\geq\gamma$ and $\underline g\in G^*_\kappa$ with $\kappa\geq \kappa_\gamma$,
clearly $d_{\underline g}( x,y )>\delta$, otherwise,
$$d(x,y)\leq \rho^{\kappa_\gamma} d(\underline g(x),\underline g(y))
	\le \rho^{\kappa_\gamma} d_{\underline g}(x,y)  <\gamma$$
which leads to a contradiction. This completes the proof of the lemma.
\end{proof}

Let us resume the proof of Theorem~\ref{thm:A}. Recall, from Lemma~\ref{le.varepsilon}  
and the fact that the semigroup action is strongly expansive, that the computation of the topological entropy of the semigroup action may be done with a well chosen, but fixed, $\varepsilon$. More precisely,
if $\delta>0$ is given by the proof of Lemma~\ref{le:R-expanding} and $0<\varepsilon<\delta$ then
$$
h_{\text{top}}(S)=\limsup_{n\to\infty}\frac1n\log\left(\frac1{p^n}\sum_{|\underline{g}|=n}\,s(\underline g,n,\varepsilon)\right).
$$
Fix one such $0<\varepsilon <\delta$. We claim that for every $\underline g\in G$ such that $|\underline{g}|=n$, the set $\text{Fix}(\underline g)$ is
$(\underline g, n, \varepsilon)$-separated. Otherwise, there would exist $P \neq Q \in \text{Fix}(\underline g)$ which were not $(\underline g, n, \varepsilon)$-separated, that is, such that $d( \underline g_m (P), \underline g_m (Q)  ) <\varepsilon$ for every
$0\le m \le n.$
If $\gamma=d(P,Q)/2$ and $k_\gamma \ge 1$ is given by Lemma~\ref{le:R-expanding}, then
$|\underline g^{\kappa_\gamma}|= n \kappa_\gamma \ge \kappa_\gamma$ and
$d_{\underline g^{\kappa_\gamma}}( P,Q  ) = d_{\underline g}(P,Q) <\varepsilon$,
which leads to a contradiction. This proves the claim. Therefore, $s(\underline g, n,\varepsilon) \geq \sharp \,\text{Fix}(\underline g)$, for every $\underline g\in G$ with $|\underline{g}|=n$, and,
consequently,
$$
\frac1{p^n}\sum_{|\underline{g}|=n}\,s(\underline g, n,\varepsilon) \geq \frac1{p^n}\sum_{|\underline{g}|=n}\,\sharp \,\text{Fix}(\underline g)
$$
and so
$$
h_{\text{top}}(S)
	= \limsup_{n\to\infty}\frac1n\log\left(\frac{1}{p^n}\sum_{|\underline g|=n}\,s(\underline g,n,\varepsilon)\right)
	\geq \limsup_{n\to\infty}\frac1n\log\left(\frac{1}{p^n}\sum_{|\underline g|=n}\,\text{Fix}(\underline g)\right)
	= \wp(S).
$$

To complete the proof  we have to show that the $\limsup$ in the definition of $\wp(S)$ is indeed a limit. First we notice that, as $G$ is finitely generated by Ruelle-expanding maps, by \cite[Theorem 16]{RoVa1}, it satisfies the periodic orbital specification property. Fix $\varepsilon \in (0,\delta)$, let $T(\varepsilon/2)\in\mathbb N$ be given by this property and take $\underline{g}\in G^*_{m+n+T(\varepsilon/2)}$. We observe that there exist $\underline a\in G_n$, $\underline b\in G^*_{T(\varepsilon/2)}$ and $\underline c\in G_m$ such that $\underline g=\underline a \;\underline b\;\underline c.$ Let $\mbox{Fix}(\underline{c})=\{P_1,\dots,P_r\}$ and  $\mbox{Fix}(\underline{a})=\{Q_1,\dots,Q_s\}$ be the sets of fixed points of $\underline c$ and $\underline{a}$, respectively. By the periodic specification property, for the semigroup elements $\underline{c},\underline{a}$ and the points $P_i\in\mbox{Fix}(\underline{c})$ and $Q_j\in\mbox{Fix}(\underline{a})$ there exists $x_{ij}\in \mbox{Fix}(\underline a\;\underline b\;\underline c)$ such that
$$
d(\underline c_\ell(x_{ij}),\underline c_{\ell}(P_i))< \frac{\varepsilon}{2} \quad \text{and} \quad
d(\underline a_u\;\underline b\;\underline c\,(x_{ij}),\underline a_{u}(Q_j))< \frac{\varepsilon}{2}
$$
for every $\ell=0,\dots,m$ and every $u=0,\dots, n.$ As the set $ \mbox{Fix}(\underline{c})$ is
$(\underline{c},m,\varepsilon)$-separated, we have $x_{i_1j_1}\not=x_{i_2j_2}$ for $(i_1,j_1)\not=(i_2,j_2)$.
This implies that $\sharp\,\mbox{Fix}(\underline{g})\geq \sharp\,\mbox{Fix}(\underline{a})\,\sharp\,\mbox{Fix}(\underline{c})$
and so,
$$
\sum_{|\underline g|=m+n+T(\varepsilon/2)}\sharp\,\mbox{Fix}(\underline{g})
        \geq\sum_{|\underline c|=m,|\underline a|=n}\sharp\,\mbox{Fix}(\underline{c})\,\sharp\,\mbox{Fix}(\underline{a})
=\left(\sum_{|\underline c|=m}\sharp\,\mbox{Fix}(\underline{c})\right)\left(\sum_{|\underline a|=m}\sharp \,\mbox{Fix}(\underline{a})\right).
$$
This yields
\begin{align}\label{eq:subadditivity}
         \frac{1}{p^{m+n+T(\varepsilon/2)}}\sum_{|\underline g|=m+n+T(\varepsilon/2)}\sharp \,\mbox{Fix}(\underline{g})
\geq \frac{1}{p^{T(\varepsilon/2)}}\left(\frac{1}{p^m}\sum_{|\underline c|=m}\sharp \,\mbox{Fix}(\underline{c})\right)
\left(\frac{1}{p^n}\sum_{|\underline a|=n}\sharp \,\mbox{Fix}(\underline{a})\right).
  \end{align}
If we denote by $a_n$ the value $\log\left(\frac{1}{p^n}\sum_{|\underline a|=n}\sharp \ ,\mbox{Fix}(\underline{a})\right)$, the inequality \eqref{eq:subadditivity} implies that $a_{m+n+T(\varepsilon/2)}\geq a_n+a_m$ for all $m, n\ge 1$.
As $T(\varepsilon/2)$ is a fixed constant, by a simple adaptation of the proof of Fekete's Lemma (\cite[Theorem 4.9]{Wa}), it follows that the sequence $(\frac{a_n}{n})_{n\in\mathbb N}$ converges to its supremum. Therefore the $\limsup$ in the definition of $\wp(S)$ may be replaced by a limit.
To complete the proof of the Theorem~\ref{thm:A} we are left to show that $\wp(S) = \log sp\,(\mathbf{L}_{n,0})_{n}.$
By Lemma~\ref{le:R-expanding}, each $\underline g\in G$ is a Ruelle-expanding map. Hence, $\sharp \,\text{Fix}(\underline g) = \mbox{deg}\,(\underline g)$ and then, taking the observable $\varphi\equiv 0$, the equality
$\wp(S) = \log sp((\mathbf{L}_{n,0})_{n})$
holds trivially by \eqref{eqL0}.

\section{The zeta function of the semigroup action}\label{sec:zeta}

Let $G$ be a semigroup generated by a finite set $G_1$ and $S: G \times M \to M$ the corresponding continuous semigroup action on a compact connected metric space $M$.

\begin{definition}\label{def:zeta-function}
The \emph{zeta function} associated to the continuous semigroup action $S: G \times M \to M$ is the formal power series
\begin{align}\label{de.number_periodic_points}
z \in \mathbb C \; \mapsto \; \zeta_S(z) =\exp \left(\sum_{n=1}^\infty \,\frac{N_n(G)}{n}\, z^n \right),
	\quad \text{where} \quad N_n(G)=\frac{1}{p^n}\sum_{|\underline g|=n}\, \sharp \, \text{Fix}(\underline g).
\end{align}
\end{definition}

We observe that this notion is a priori different from the one introduced by Artin-Mazur in \cite{AM}. In the particular case of $G$ being generated by $G_1^*=\{f \}$, then $G_n^* \subseteq \{f, f^2, \dots, f^n\}$ and we get
$$N_n(G)=\sum_{\{j \colon f^j \in G_n^*\}} \sharp \,\text{Fix}(f^j)
\quad \text{and}\quad
\zeta_S(z)=\exp\left(\sum_{n=1}^\infty \,\frac{\sum_{\{j \colon f^j \in G_n^*\}} \sharp \,\text{Fix}(f^j)}{n} \,z^n\right)
$$
while the dynamical Artin-Mazur's zeta function computes
$$\zeta_f(z)=\exp\left(\sum_{n=1}^\infty \,\frac{\sharp \,\text{Fix}(f^n)}{n} \,z^n\right).$$
These are different for instance if $f$ is a finite order element in $G$ as a rational rotation on $S^1$.
However, in the case of a topologically mixing Ruelle-expanding map $f:M \to M$, we have $G_n^*=\{f^n\}$ and, therefore, the zeta function $\zeta_S$ coincides with $\zeta_f$. As mentioned before (cf. Subsection~\ref{zeta function}), in this context, the periodic points have a definite exponential growth
$$\lim_{n\to\infty} \frac1n \log \,\sharp\, \text{Fix}(f^n)=h_{\text{top}}(f)=-\log(\rho_f)$$
where $\rho_f$ is the radius of convergence of the zeta function $\zeta_f$. We refer the reader to \cite{Bz02} for an account on random zeta functions.

\subsection{Proof of Theorem~\ref{thm:B}}

The function $\zeta_S$ we associate to the semigroup $G$ generated by a finite set $G_1$, with $G_1^*$ a finite set of Ruelle-expanding maps, is linked to the notion of annealed zeta function introduced in \cite{Baladi} in the context of random families of $C^r$ expanding maps, $r > 1$.
The aim of this section is to show that, when we consider Ruelle-expanding maps and the random walk $R_{\underline p}$, the annealed zeta function is rational and its radius of convergence is $\exp\,(-h_{\text{top}}(S))$.

We will start estimating the radius $\rho_S$ of convergence of $\zeta_S$ and relating it with $h_{\text{top}}(S)$. We first notice that, as $\lim_{n\to\infty}\,\sqrt[n]{n}=1$, then
\begin{equation}\label{def:radiuszeta}
\frac{1}{\rho_S} = {\limsup_{n\to\infty}\sqrt[n]{\frac{N_n(G)}{n}}}
	 = {\limsup_{n\to\infty}\sqrt[n]{\frac{1}{p^n}\sum_{|\underline{g}|=n}\,\sharp \,\text{Fix}(\underline g)}}.
\end{equation}
On the other hand,
$$\limsup_{n\to\infty}\sqrt[n]{\frac{1}{p^n}\sum_{|\underline g|=n}\,\sharp \,\text{Fix}(\underline g)}=\limsup_{n\to\infty}\sqrt[n]{\max \{ N_n(G),1\}}=\exp (\wp(S)).$$
Consequently, $\rho_S=\exp (-\wp(S))$. Thus, whenever $\wp(S)>0$, the zeta function $\zeta_S$ has a positive radius of convergence, meaning it is well defined in $\{ z\in \mathbb C \colon |z|<\exp (-\wp(S))\}$. Under the assumptions of Theorem~\ref{thm:A}, one also has
$\rho_S=\exp(-h_{\text{top}}(S)).$

We are left to prove that the zeta function of a semigroup of Ruelle-expanding maps is rational. We start showing that, under the assumptions of Theorem~\ref{thm:B}, the skew-product
\begin{equation}\label{de.skew-product}
\begin{array}{rccc}
\mathcal{F}_G : & \Sigma_p^+  \times M & \to & \Sigma_p^+  \times M \\
	& (\omega,x) & \mapsto & (\sigma(\omega), g_{\omega_1}(x))
\end{array}
\end{equation}
where $\omega=(\omega_1,\omega_2, \dots)$, has the following properties.

\begin{lemma}\label{le.skew expanding mixing}
The map $\mathcal{F}_G$ is Ruelle-expanding and topologically mixing.
\end{lemma}

\begin{proof} Denote by $d_M$ and $d_{\sum}$ the metrics in $M$ and $\Sigma_p^+$, respectively. We are considering in $\Sigma_p^+  \times M$ the product topology, which is metrizable; its topology is given, for instance, by the metric
$D((\omega^0,x_0),(\omega^1,x_1))=\max\,\{d_M(x_0,x_1), d_{\sum}(\omega^0, \omega^1)\}.$
As $\sigma$ and each $g_i \in G_1^*$ are Ruelle-expanding, there exist positive constants $c_\sigma$ and $c_i$, for $i \in \{1,\dots,p\}$, such that
\begin{eqnarray*}
x,y \in M,\,\, x\neq y, \,\,g_i(x)=g_i(y) \quad &\Rightarrow& \quad d_M(x,y)>c_i \\
\omega^0, \omega^1 \in  \Sigma_p^+,\,\, \omega^0 \neq \omega^1, \,\,\sigma(\omega^0)=\sigma(\omega^1) \quad &\Rightarrow& \quad d_{\sum}(\omega^0, \omega^1)>c_\sigma. \\
\end{eqnarray*}
Let $(\omega^0,x_0)\neq (\omega^1,x_1)$ be such that $\mathcal{F}_G((\omega^0,x_0))= \mathcal{F}_G((\omega^1,x_1))$, that is, $\sigma(\omega^0)=\sigma(\omega^1)$ and $g_{\omega^0_1}(x_0)=g_{\omega^1_1}(x_1)$. Then, either $\omega^0 \neq \omega^1$, in which case we have
$D((\omega^0,x_0),(\omega^1,x_1))> c_\sigma;$
or else $\omega^0 = \omega^1$, and then
$D((\omega^0,x_0),(\omega^1,x_1)) > c_{\omega^0_1}.$
Therefore, if $c=\min\,\{c_\sigma, c_1, c_2, \dots,c_d\}$, then
$$(\omega^0,x_0)\neq (\omega^1,x_1), \,\,\mathcal{F}_G((\omega^0,x_0))= \mathcal{F}_G((\omega^1,x_1)) \quad \Rightarrow \quad D((\omega^0,x_0), (\omega^1,x_1)> c.$$

Let us now address the second property that characterizes Ruelle-expanding maps. Take $r$ and $\rho$ as in (\ref{eq:constants-Ru}), $r_\sigma$ and $\rho_\sigma$ the corresponding values for $\sigma$ due to its expanding nature (see Definition~\ref{de.Ruelle-expanding} and the first example in Subsection~\ref{ex.unilateral-shift}), and set $r_{\mathcal{F}_G}=\min\,\{r, r_\sigma\}$ and $\rho_{\mathcal{F}_G}=\min\,\{\rho, \rho_\sigma\}$. Then, given $(\omega, x) \in \Sigma_p^+ \times M$ and $(\beta, a)$ such that $\mathcal{F}_G((\beta, a))=(\omega, x)$, there exist maps $\varphi_{\omega_1}: B_{r}(x) \to M$ and $\varphi_{\sigma}: B_{r_\sigma}(\omega) \to \Sigma_p^+$ such that $\varphi_{\omega_1}(x) = a$, $\varphi_{\sigma}(\omega) = \beta$, $g_{\omega_1} \circ \varphi_{\omega_1} = id$, $\sigma \circ \varphi_{\sigma} = id$ and
\begin{eqnarray*}
d_M(\varphi_{\omega_1}(z), \varphi_{\omega_1}(w)) &\leq& \rho d_M(z,w), \quad \forall z,w \in  B_{r}(x) \\
d_{\sum}(\varphi_{\sigma}(s), \varphi_{\sigma}(t)) &\leq& \rho_\sigma d_\sigma(s,t), \quad \forall s,t \in  B_{r_\sigma}(\omega).
\end{eqnarray*}
Therefore, if $R>0$ is such that, for all $(\gamma,b) \in \Sigma_p^+ \times M$, we have $B_R((\gamma,b)) \subset B_{r_{\mathcal{F}_G}}(\gamma) \times B_{r_{\mathcal{F}_G}}(b)$, then the map $\varphi_{\sigma} \times \varphi_{\omega_1}: B_R((\omega,x)) \to \Sigma_p^+ \times M$ satisfies $\varphi_{\sigma} \times \varphi_{\omega_1}(\omega,x) = (\beta, a)$,
$\mathcal{F}_G \circ (\varphi_{\sigma} \times \varphi_{\omega_1}) = id$ and
\begin{eqnarray*}
D(\varphi_{\sigma} \times \varphi_{\omega_1}(s,z), \varphi_{\sigma} \times \varphi_{\omega_1}(t,w)) &=& \max\{d_M(\varphi_{\omega_1}(z), \varphi_{\omega_1}(w)), d_{\sum}(\varphi_{\sigma}(s), \varphi_{\sigma}(t))\}) \\
&\leq& \rho_{\mathcal{F}_G}\max\{d_M(z,w), d_{\sum}(s,t)\} \\
&=& \rho_{\mathcal{F}_G}D((s,z),(t,w)) \quad \quad \forall s,t \in  B_{r_{\mathcal{F}_G}}(\omega) \times B_{r_{\mathcal{F}_G}}(x).
\end{eqnarray*}
This ends the proof that $\mathcal{F}_G$ is Ruelle-expanding.

We now proceed showing that $\mathcal{F}_G$ is topologically mixing. Consider a non-empty open subset $\mathcal{W}$ of $\Sigma_p^+ \times M$ and take a cylinder $U=C(1; a_1\,a_2\,\dots,a_k)$ and an open set $V$ of $M$ such that $U\times V \subset \mathcal{W}$. As the maps $\sigma$ and $g_{a_k}\dots g_{a_1}$ are topologically mixing and Ruelle-expanding, there exist positive integers $m_U$ and $m_V$ such that $\sigma^{\ell}(U)=\Sigma_p^+$ and $(g_{a_k}\dots g_{a_1})^{\ell}=M$, for all $\ell \geq m=\max\,\{m_U, m_V\}$. Hence, for all $\ell \geq m$ we have
$$\mathcal{F}_G^\ell (U \times V)=\left(\sigma^\ell(U), \bigcup_{\omega \in U}\,f_{\omega}^\ell(V)\right)= \Sigma_p^+ \times M$$
since $V$ contains all the sequences of $\sum_{d}^+$ whose $k$ first entries are $a_1\,a_2\,\dots,a_k$, in particular those which start with this block repeated $\ell$ times for every $\ell \in \mathbb{N}$.
\end{proof}

\begin{corollary}\label{le.rational}
The Artin-Mazur zeta function of $\mathcal{F}_G$ is rational.
\end{corollary}

\begin{proof} After Lemma~\ref{le.skew expanding mixing}, we have just to summon Subsection~\ref{zeta function}.
\end{proof}

Given $n \in \mathbb{N}$, let $\sharp\,\text{Per}_n(\mathcal{F}_G)$ be the number of periodic points with period $n$ of $\mathcal{F}_G$.

\begin{lemma}\label{le.periodicpoints_skew_semigroup}
$\sharp\,\text{Per}_n(\mathcal{F}_G)=p^n \times N_n(G).$
\end{lemma}

\begin{proof} We observe that $(\omega,x) \in \text{Per}_n(\mathcal{F}_G)$ is and only if $\sigma^n(\omega)=\omega$ and $f_\omega^n(x)=x$. That is, $\omega$ is a periodic point with period $n$ of $\sigma$, whose set has cardinal $p^n$, and $x\in \text{Fix}(f_\omega^n)$. Thus
$$
\sharp\,\text{Per}_n(\mathcal{F}_G)
	= \sum_{\sigma^n(\omega)=\omega} \sharp \text{Fix}(f_\omega^n)
	= \sum_{|\underline g|=n} \sharp \text{Fix}(\underline g)
	= p^n \times N_n(G).
$$
\end{proof}

\begin{corollary}
The zeta function of the semigroup is rational.
\end{corollary}

\begin{proof} Let $\zeta_{\mathcal{F}_G}$ be the Artin-Mazur zeta function of the skew-product $\mathcal{F}_G$. Given
$z\in \mathbb C$,
\begin{align*}
\zeta_S(z)  & = \exp \Big( \sum_{n=1}^{+\infty}\,\frac{N_n(G)}{n}\,z^n \Big)
	= \exp \Big(\sum_{n=1}^{+\infty}\,\frac{\sharp\,\text{Per}_n(\mathcal{F}_G)}{n \times p^{n}}\,z^n \Big) \\
	& = \exp \Big( \sum_{n=1}^{+\infty}\,\frac{\sharp\,\text{Per}_n(\mathcal{F}_G)}{n}\,\left(\frac{z}{p}\right)^n \Big)
	= \zeta_{\mathcal{F}_G}\left(\frac{z}{p}\right).
\end{align*}
\end{proof}

\section{Intrinsic objects \emph{vs} skew-product dynamics}\label{se.entropy}

In this section we will establish a bridge between topological Markov chains and semigroup actions in what concerns the relation between the notions of fibered and relative topological entropies and the concepts of annealed and quenched topological pressures.

\subsection{Thermodynamic formalism for the skew-product}\label{sec:skewp}

There have been several approaches to study the thermodynamic formalism of skew-product dynamics:
 (i) Ruelle expanding skew-product maps, (ii) fibered entropy for factor maps, (iii) quenched and annealed equilibrium states
 in random dynamical systems and (iv) relative measures for skew-products. We collect some of them here to be later compared with our results.

\subsubsection*{\emph{\textbf{Topological entropy of the Ruelle expanding skew-product}}}\label{skew-productR}

Since the skew product $\cF_G$ is a Ruelle-expanding map (cf. Lemma~\ref{le.skew expanding mixing}), for any H\"older continuous potential $\varphi$ on $\Sigma_p^+ \times M$ the map $\cF_G$ admits a unique equilibrium state
as described in Section~\ref{sec:Ruelle-expanding}. In particular, if $\varphi\equiv 0$, there is a unique measure $\mu_{\underline m}$ of maximal entropy of $\cF_G$, which is equally distributed along the $\sum_{i=1}^p \deg(g_i)$ elements  of the natural Markov partition $\mathcal Q$ on $\Sigma_p^+\times M$ and may be computed by the limit process described in
Section~\ref{sec:Ruelle-expanding}. From \cite{S00}, we also know that the projection of $\mu_{\underline m}$ in $\Sigma^+_p$ is $\eta_{\underline m}$, where
\begin{equation}\label{eq:m}
\underline m = \Big(\frac{\deg\,(g_1)}{\sum_{k=1}^p\,\deg\,(g_k)}, \frac{\deg\,(g_2)}{\sum_{k=1}^p\,\deg\,(g_k)}, \ldots,\frac{\deg\,(g_p)}{\sum_{k=1}^p\,\deg\,(g_k)}\Big).
\end{equation}
Moreover, the topological entropy of the skew-product $\mathcal{F}_G$
is given by
\begin{equation}\label{formula.Bufetov}
h_{\text{top}}(\mathcal{F}_G)= h_{\text{top}}(S) + \log p
\end{equation}
(cf. \cite{Bufetov}), so
$$
h_{\text{top}}(\mathcal{F}_G) = \log \Big(\sum_{i=1}^p \deg\,(g_i)\Big)
	\quad \text{and}\quad
h_{\text{top}}(S) = \log \Big(\frac{\sum_{i=1}^p \deg\,(g_i)}{p}\Big).
$$
Notice that the last equality for $h_{\text{top}}(S)$ depends only on the ingredients that set up the semigroup action
(in particular, it does not explicitly display the skew-product).
More generally (see \cite{Ru2}), if $\varphi$ is piecewise constant along the $\sum_{i=1}^p \deg(g_i)$ elements of Markov partition $\mathcal Q$,
then there exists a unique equilibrium state $\mu_\varphi$ for $\cF_G$ with respect to $\varphi$, it is a Bernoulli measure and satisfies
\begin{equation}\label{eq:classic}
P_{\text{top}}(\cF_G,\varphi)=\log\Big( \sum_{\substack{Q\in \mathcal Q}} e^{\varphi(Q)}\Big)
\end{equation}
and, for every $Q\in \mathcal Q$,
\begin{equation}\label{eq:classic}
	\mu_\varphi(Q)= \frac{e^{\varphi(Q)}}{ \sum_{\substack{Q\in \mathcal Q}} e^{\varphi(Q)}}.
\end{equation}

\subsubsection*{\emph{\textbf{Fibered entropy of the skew-product}}}\label{skew-product}

Following \cite{LW77}, consider the skew-product $\mathcal{F}_G$ and the projection on the first coordinate, say $\pi : \Sigma_p^+  \times M \to  \Sigma_p^+$, $\pi(\omega,x)=\omega$. We say that a subset $E$ of $\pi^{-1}(\omega)$ is $(n,\varepsilon)$-separated if
there exists $i \in \{0, \ldots, n\}$ such that
$d(g_{\omega_{i}}\dots g_{\omega_1}(x), g_{\omega_{i}}\dots g_{\omega_1}(y))\geq \varepsilon$
where $\underline{g}=g_{\omega_{n}}\dots g_{\omega_1} \in G$ and $|\underline{g}|=n$. Therefore, if $s(n,\varepsilon,\pi^{-1}(\omega))$ is the maximal cardinality of a $(n,\varepsilon)$-separated subset of $\pi^{-1}(\omega)$, then
$s(n,\varepsilon,\pi^{-1}(\omega))= s(\underline{g},n,\varepsilon).$
By \cite{LW77}, given a $\sigma$-invariant probability measure $\eta$ on $\Sigma_p^+$, then the map
$$
\omega \mapsto h_{\text{top}}(\mathcal{F}_G,\pi^{-1}(\omega)) :=\lim_{\varepsilon\to0}\limsup_{n\to\infty}\frac{1}{n}\log s(n,\varepsilon,\pi^{-1}(\omega))
$$
is measurable and
\begin{equation}\label{LW}
\sup_{\{\mu \colon {\mathcal{F}_G}_*(\mu)=\mu,\,\pi_*(\mu)\,=\,\eta \}}\,h_\mu(\mathcal{F}_G)= h_\eta(\sigma) + \int_{\Sigma_p^+}\,h_{\text{top}}(\mathcal{F}_G,\pi^{-1}(\omega))\,d\eta(\omega).
\end{equation}
We will refer to $h_{\text{top}}(\mathcal{F}_G,\pi^{-1}(\omega))$ as the \emph{relative entropy} on the fiber $\pi^{-1}(\omega)$; the \emph{fibered entropy} of the semigroup action $S$ with respect to $\eta$ is $\int_{\Sigma_p^+}\,h_{\text{top}}(\mathcal{F}_G,\pi^{-1}(\omega))\,d\eta(\omega)$.

\subsubsection*{\textbf{\emph{Quenched and annealed equilibrium states for the skew-product 
}}}\label{sec:random}

We now follow \cite{Baladi} to define quenched and annealed equilibrium states.
Given a continuous potential $\varphi: \Sigma_p^+ \times M \to \mathbb R$ and a probability measure $\underline a$
on $\{1, \dots, p\}$, the \emph{annealed topological pressure} of $\cF_G$ with respect to $\varphi$ and $\underline a$
is defined as
\begin{align}\label{def:annealed.entropy}
P_{\text{top}}^{(a)}(\cF_G,\varphi, \underline a)
	& = \sup_{\{\mu \colon {\mathcal{F}_G}_*\mu=\mu\}}
		\Big\{ h_\mu(\cF_G) - h_{\pi_\mu}(\sigma)  + h^{\underline a}(\pi_\mu)
		+ \int \varphi(\omega,x) \, d\mu(\omega,x) \Big\} \nonumber
	\end{align}
where $\omega=(\omega_1,\omega_2,\ldots)$,  $\pi_\mu=\pi_*(\mu)$ is the marginal of $\mu$ in $\Sigma^+_p$ and the \emph{entropy per site} $h^{\underline a}(\pi_\mu)$ 
with respect to $\eta_{\underline a}$ is given by
$$h^{\eta_{\underline a}}(\pi_\mu) = - \int_{\Sigma_p^+} \log \psi_{\pi_\mu}(\omega) \; d\pi_\mu(\omega) =  \int_{\Sigma_p^+} -\psi_{\pi_\mu}(\omega) \log \psi_{\pi_\mu}(\omega) \; d{\underline a}(\omega_1) \,\,d\pi_\mu (\sigma (\omega))$$
if $d\pi_\mu(\omega_1,\omega_2, \dots) \ll d{\underline a}(\omega_1) \,\,d\pi_\mu(\omega_2,\omega_3, \dots)$ and $\psi_{\pi_\mu}:=\frac{d\pi_\mu}{d{\underline a} \,\,d\pi_\mu \circ \sigma}$
denotes the Radon-Nykodin derivative of $\pi_\mu$ with respect to $\underline a \times \pi_\mu \circ \sigma$;
and $h^{\underline a}(\pi_\mu) = -\infty$ otherwise. We recall from \cite[page 676]{Baladi} that
$$h^{\underline a}(\pi_\mu)=0 \quad \Leftrightarrow \quad \pi_\mu=\eta_{\underline a}.$$
According to \cite[Equation (2.28)]{Baladi}, the annealed pressure can also be evaluated by
\begin{align}
P_{\text{top}}^{(a)}(\cF_G,\varphi, \underline a)
		& = \sup_{\{\mu \colon {\mathcal{F}_G}_*\mu=\mu\}}
		\Big\{ h_\mu(\cF_G) + \int_{\Sigma^+_p \times M} \log \big( \, \underline a (\omega_1) e^{\varphi(\omega,x)} \big)\, d\mu(\omega,x) \Big\}.
\end{align}
The \emph{quenched topological pressure} of $\cF_G$ with respect to $\varphi$ and $\underline a$
is defined as
\begin{equation}\label{def:quenched.entropy}
P_{\text{top}}^{(q)}(\cF_G,\varphi, \underline a)
	= \sup_{\{\mu \colon {\mathcal{F}_G}_*\mu=\mu, \, \pi_\mu=\eta_{\underline a}\}}
		\Big\{ h_\mu(\cF_G) - h_{\eta_{\underline a}}(\sigma) + \int \varphi(\omega,x) \, d\mu(\omega,x) \Big\}.
\end{equation}
It follows from the definition that we always have $P_{\text{top}}^{(a)}(\cF_G,\varphi, \underline a) \ge P_{\text{top}}^{(q)}(\cF_G,\varphi, \underline a)$.

An $\cF_G$-invariant probability measure is said to be an \emph{annealed (resp. quenched) equilibrium state for $\cF_G$ with respect to $\varphi$ and $\underline a$} if it
attains the supremum in equation~\eqref{def:annealed.entropy} (resp. equation~\eqref{def:quenched.entropy}).
In the case of finitely generated semigroups of $C^2$ expanding maps, there exists a unique quenched and a unique annealed equilibrium state for every H\"older continuous observable $\varphi$ and every $\underline a$, and
they exhibit an exponential decay of correlations (cf. \cite{Baladi}). For instance, for a semigroup $G$ with generators $G_1=\{id, g_1, \dots, g_p\}$ where each $g_i$ is a $C^2$-smooth expanding map, consider the H\"older potential $\varphi(\omega,x)=- \log |\det Dg_{\omega_1}(x)|$. Its annealed equilibrium state $\mu^{(a)}_{\varphi,\underline a}$ was described in \cite[Proposition 2]{Baladi} and is also the quenched equilibrium state for this potential.

\begin{remark}\label{rem:relative topological pressure}
Given a continuous potential $\varphi : \Sigma_p^+\times M \to \mathbb R$ and a $\sigma$-invariant probability measure $\eta$ on $\Sigma_p^+$, the notion of relative pressure of $\varphi$ on the fiber $\pi^{-1}(\omega)$, denoted by $P_{\text{top}}(\mathcal{F}_G, \varphi, \pi^{-1}(\omega))$, was studied in \cite[Section 2]{LW77}. In this reference it was shown that the relative pressure satisfies the following relative variational principle:
\begin{align}\label{eq:relative.variational.principle}
\int_{\Sigma_p^+} P_{\text{top}}(\mathcal{F}_G, \varphi, \pi^{-1}(\omega)) \,d\eta(\omega)
	& = \sup_{\{\mu \colon {\mathcal{F}_G}_*(\mu)=\mu, \, \pi_*(\mu)=\eta\}} \, \Big\{ h_\mu(\cF_G) - h_{\eta}(\sigma) + \int \varphi \,d\mu \Big\}.
\end{align}
Hence, a quenched equilibrium state for $\cF_G$ with respect to $\varphi$ and $\underline a$ is also an $\cF_G$-invariant probability measure $\mu^{(q)}_{\varphi,\underline a}$ satisfying
\begin{align}\label{rem.quenched}
\int_{\Sigma_p^+} P_{\text{top}}(\mathcal{F}_G, \varphi, \pi^{-1}(\omega)) \,d\eta_{\underline a}(\omega)
	& =  h_{\mu^{(q)}_{\varphi,\underline a}}(\cF_G) - h_{\eta_{\underline a}}(\sigma) + \int \varphi \; d\mu^{(q)}_{\varphi,\underline a}.
\end{align}
\end{remark}

\begin{example} Let  $G$ be a semigroup with generators $G_1=\{id, g_1, \dots, g_p\}$, where each $g_i$ is a $C^2$-smooth expanding map. For the potential $\varphi\equiv 0$ we will analyze how the annealed and quenched pressures vary with $\underline a$.
When $\underline a= \underline p $, we obtain
\begin{align}\label{eq:annealed_top-entropyS}
P_{\text{top}}^{(a)}(\cF_G, 0 , \underline p) &= \sup_{\{\mu \colon {\mathcal{F}_G}_*\mu=\mu\}}
		\Big\{ h_\mu(\cF_G) + \int \log \big(\,\underline p (\omega_1)\big)\, d\mu(\omega,x) \Big\}\nonumber\\
& = \sup_{\{\mu \colon {\mathcal{F}_G}_*\mu=\mu\}} \Big\{ h_\mu(\cF_G) - \log p \Big\}
 = h_{\text{top}}(\cF_G) -\log p = h_{\text{top}}(S).
\end{align}
The corresponding annealed equilibrium state $\mu^{(a)}_{\underline p}$ is the measure of maximal entropy $\mu_{\underline m}$ of $\cF_G$.
Moreover, by \cite[Equation (2.28)]{Baladi}, for any non-trivial probability vector $\underline a$ and the corresponding annealed equilibrium state $\mu^{(a)}_{\underline a}$ we have
$$h^{\underline a}(\pi_{\mu^{(a)}_{\underline a}}) = h_{\pi_{\mu^{(a)}_{\underline a}}}(\sigma)  + \int_{\Sigma^+_p \times M} \log \big(\,\underline a (\omega_1) \big)\, d\mu^{(a)}_{\underline a}(\omega,x).$$
So, when $\underline a=\underline p=(\frac1p, \dots, \frac1p)$,
\begin{equation}\label{eq:entropy-per-site-p}
h^{\underline p}(\pi_{\mu^{(a)}_{\underline p}}) = h_{\pi_{\mu^{(a)}_{\underline p}}}(\sigma)  + \int_{\Sigma^+_p \times M} \log \big(\,\underline p (\omega_1) \big)\, d\mu^{(a)}_{\underline p}(\omega,x) = h_{\eta_{\underline m}}(\sigma)  - \log p = h_{\eta_{\underline m}}(\sigma)  - h_{\eta_{\underline p}}(\sigma).
\end{equation}
Concerning the quenched operator, we get
\begin{align*}
& P_{\text{top}}^{(q)}(\cF_G, 0, \underline p)
	= \sup_{\{\mu \colon {\mathcal{F}_G}_*\mu=\mu, \, \pi_\mu=\eta_{\underline p}\}}  \; 		
		\Big\{ h_\mu(\cF_G) \Big\} - \log p.
\end{align*}
Yet, as $\pi(\mu_{\underline m})=\eta_{\underline m}$ and ${\underline m}\neq {\underline p}$ (cf. \cite{S00}), the quenched equilibrium state differs from $\mu^{(a)}_{\underline p}$.
In general, for a non-trivial probability vector $\underline a$, we have
\begin{align}\label{eq:equalpress}
 P_{\text{top}}^{(a)}(\cF_G, 0 , \underline a)
 	& = \sup_{\{\mu \colon {\mathcal{F}_G}_*\mu=\mu\}}
		\Big\{ h_\mu(\cF_G) + \int_{\Sigma^+_p \times M} \log \big(\,\underline a (\omega_1)\big)\, d\mu(\omega,x) \Big\}\nonumber \\
	& = P_{\text{top}}(\cF_G, \varphi_{\underline a})
\end{align}
where $\varphi_{\underline a} : \Sigma_p^+ \times M \to \mathbb R$ is the locally constant potential given by
$$ \varphi_{\underline a}(\omega,x)= \log \underline a(\omega_0)$$
and the quenched pressure is given by
\begin{align*}
& P_{\text{top}}^{(q)}(\cF_G, 0, \underline a)
	= \sup_{\{\mu \colon {\mathcal{F}_G}_*\mu=\mu, \, \pi_\mu=\eta_{\underline a}\}}  \; 		
		\Big\{ h_\mu(\cF_G) \Big\} - h_{\eta_{\underline a}}(\sigma).
\end{align*}
Therefore, a quenched equilibrium state $\mu^{(q)}_{\underline a}$ for $\varphi \equiv 0$ and $\underline a$ satisfies
$\pi_*(\mu^{(q)}_{\underline a})=\eta_{\underline a}$ and
 \begin{equation}\label{eq:quenched-a}
 h_{\mu^{(q)}_{\underline a}}(\cF_G) = \sup_{\{\mu \colon {\mathcal{F}_G}_*\mu=\mu, \, \pi_\mu=\eta_{\underline a}\}}  \; 				 h_\mu(\cF_G).
 \end{equation}
In particular, when $\underline a =\underline m$, we conclude that
\begin{equation}\label{eq:quenched-annealed}
\mu^{(q)}_{\underline m}=\mu_{\underline m}= \mu^{(a)}_{\underline p}
	\quad\text{and}\quad
	 P_{\text{top}}^{(q)}(\cF_G, 0, \underline m) =  h_{\text{top}}(\cF_G) - h_{\eta_{\underline m}}(\sigma).
\end{equation}
\end{example}

\subsubsection*{\emph{\textbf{Relative measures for the skew-product}}}\label{subset:Sumi}

Another proposal to study skew-product dynamics and semigroup actions was explored in \cite{S00, SU09}.
In ~\cite[Theorem 1.3]{S00} it is shown that, if $\eta_{\underline a}$ is the Bernoulli probability measure on $\Sigma_p^+$ determined by the vector $\underline a=(a_1, a_2, \ldots,a_p)$, such that $a_k>0$ for all $k = 1, \dots, p$ and $\sum_{k=1}^{p}\,a_k = 1$, there exists a self-similar probability measure
$\mu_{\underline a}$ which is invariant under the skew-product $\mathcal{F}_G$, whose projection to the base space is $\pi_*(\mu_{\underline a})=\eta_{\underline a}$ and which satisfies
\begin{equation}\label{eq.sumi1}
h_{\mu_{\underline a}}(\mathcal{F}_G) = \sup_{\{\mu \colon {\mathcal{F}_G}_*(\mu)=\mu,\,\pi_*(\mu)\,=\,\eta_{\underline a}\}}\,h_\mu(\mathcal{F}_G)
\end{equation}
and
\begin{equation}\label{eq.sumi}
h_{\mu_{\underline a}}(\mathcal{F}_G) = h_{\eta_{\underline a}}(\sigma) + \sum_{k=1}^p\, a_k \,\log \deg\,(g_k).
\end{equation}
where $\deg\,(g_k)$ stands for the degree of the map $g_k$. In the particular case of $\underline a=\underline m$,
the measure $\mu_{\underline m}$ is the unique probability of maximal entropy of the skew-product $\mathcal{F}_G$ (cf. Subsection~\ref{skew-productR}), and a simple computation yields that
\begin{equation}\label{eq:topsumi}
h_{\text{top}}(\mathcal{F}_G) = \log\Big(\sum_{i=1}^p \deg \,(g_i)\Big).
\end{equation}

\subsection{Topological entropy of the semigroup action with respect to a random walk}

In what follows we compare the notions of entropy for the semigroup action 
with the previous notions of fibered, quenched, annealed and relative pressures.

\subsubsection*{\textbf{\emph{Fibered entropy for the symmetric random walk}}}\label{skew-product1}

We first study the fibered entropy in the case of a symmetric random walk, that is, when each semigroup generator receives the same weight $\frac{1}{p}$. Notice that, in the case of a Bernoulli measure $\eta=\eta_{\underline a}$, the
relations ~\eqref{def:quenched.entropy} and ~\eqref{LW} imply that the fibered entropy can be computed as a quenched pressure by
$$
\int_{\Sigma_p^+}\,h_{\text{top}}(\mathcal{F}_G,\pi^{-1}(\omega))\,d\eta_{\underline a}(\omega)=P_{\text{top}}^{(q)}(\cF_G,0, \underline a).
$$
If $\eta= \eta_{\underline p}$, the formula (\ref{LW}) becomes
\begin{equation}\label{LW2}
\sup_{\{\mu \colon {\mathcal{F}_G}_*(\mu)=\mu,\,\pi_*(\mu)\,=\,\eta_{\underline p}\}} \,h_\mu(\mathcal{F}_G) = \log p + \int_{\Sigma_p^+}\,h_{\text{top}}(\mathcal{F}_G,\pi^{-1}(\omega))\,d\eta_{\underline p}(\omega).
\end{equation}
Thus, taking into account that $h_{\text{top}}(\mathcal{F}_G) = \sup_\mu \,h_\mu(\mathcal{F}_G)$ (cf. \cite{G71}), we conclude that
\begin{equation}\label{Led-Wal}
\log p + \int_{\Sigma_p^+}\,h_{\text{top}}(\mathcal{F}_G,\pi^{-1}(\omega))\,d\eta_{\underline p}
	\leq h_{\text{top}}(\mathcal{F}_G)
	\leq \log p + \sup_{\omega \in \Sigma_p^+}\,h_{\text{top}}(\mathcal{F}_G,\pi^{-1}(\omega)).
\end{equation}
Besides, \eqref{formula.Bufetov} indicates that $h_{\text{top}}(\mathcal{F}_G)= h_{\text{top}}(S) + \log p$. Together with (\ref{Led-Wal}), this implies
$$ \int_{\Sigma_p^+}\,h_{\text{top}}(\mathcal{F}_G,\pi^{-1}(\omega))\,d\eta_{\underline p} \leq h_{\text{top}}(S) \leq  \sup_{\omega \in \Sigma_p^+}\,h_{\text{top}}(\mathcal{F}_G,\pi^{-1}(\omega)).$$
So, it makes sense to ask under what conditions we have $h_{\text{top}}(S)=\int_{\Sigma_p^+}\,h_{\text{top}}(\mathcal{F}_G,\pi^{-1}(\omega))\,d\eta_{\underline p}$.

\begin{proposition}\label{le:hfibered}
Let $G_1=\{id, g_1,\cdots,g_p\}$, $p \geq 2$, be a finite set of expanding maps in $End^2(M)$, $G$ be the semigroup generated by $G_1$ and $\mathcal{F}_G :  \Sigma_p^+  \times M  \to  \Sigma_p^+  \times M$ be the corresponding skew-product. Then the following conditions are equivalent:
\begin{enumerate}
\item $h_{\text{top}}(S)=\int_{\Sigma_p^+}\,h_{\text{top}}(\mathcal{F}_G,\pi^{-1}(\omega))\,d\eta_{\underline p}$.
\item $\eta_{\underline p}=\pi_*(\mu_{\underline m})$, where $\mu_{\underline m}$ is the unique maximal entropy measure for $\cF_G$.
\item The degrees of the maps $g_i$ are the same for all $1\le i \le p$.
\end{enumerate}
Moreover, if any of these conditions holds, then
$$h_{\text{top}}(S)=\int_{\Sigma_p^+}\, \log \deg \,(g_\omega) \,d\eta_{\underline p}(\omega).$$
\end{proposition}

\begin{proof}
Recall that $\mathcal{F}_G$ is a topologically mixing Ruelle-expanding map (cf. Lemma~\ref{le.skew expanding mixing}). Hence, taking $\varphi\equiv 0$ in Theorem~\ref{te.Ruelle}, we deduce that $\cF_G$ has a unique maximal entropy measure $\mu_{\underline m}$. Furthermore, it follows from the variational principle for $\cF_G$ and the relations \eqref{LW2} and \eqref{formula.Bufetov} that
\begin{align*}
h_{\text{top}}(S)  = h_{\text{top}}(\mathcal{F}_G)-\log p
	& \ge  \sup_{\{\mu \colon {\mathcal{F}_G}_*(\mu)=\mu,\,\pi_*(\mu)\,=\,\eta_{\underline p}\}}\,[h_\mu(\mathcal{F}_G) - \log p]
	 \\ &=  \int_{\Sigma_p^+}\,h_{\text{top}}(\mathcal{F}_G,\pi^{-1}(\omega))\,d\eta_{\underline p}(\omega).
\end{align*}
Clearly the previous inequality becomes an equality if an only if $\pi_*(\mu_{\underline m})=\eta_{\underline p}$, which proves that (1) is equivalent to (2). The remaining of the proof relies on the property of $\mu_{\underline m}$ stated in \eqref{eq:m}. It implies that $h_{\text{top}}(S)=\int_{\Sigma_p^+}\,h_{\text{top}}(\mathcal{F}_G,\pi^{-1}(\omega))\,d\eta_{\underline p}$ if and only the degrees of the maps $g_i$ are the same for all $1\le i \le p$. This proves that (2) is equivalent to (3). Finally, as $\eta_{\underline p}$ is a Bernoulli measure, $h_{\eta_{\underline p}} (\sigma)=\log p$
and $h_{\mu_{\underline m}}(\cF_G)=h_{\text{top}}(\mathcal{F}_G)$, then
$$h_{\text{top}}(S) =  \frac1p \, \sum_{k=1}^p\, \log \deg\,(g_k) = \int_{\Sigma_p^+} \, \log \deg \,(g_{\omega_1}) \,d\eta_{\underline p}(\omega).$$
This completes the proof of the proposition.
\end{proof}

\subsubsection*{\textbf{\emph{Relative entropy for non-symmetric random walks}}}\label{symmetry}

The notion $h_{\text{top}}(S)$ of topological entropy of a semigroup action $S$ depends on our choice of the Bernoulli equally distributed probability measure $\eta_{\underline p}$ on $\Sigma_p^+$.
If, instead of $\eta_{\underline p}$, we take another $\sigma$-invariant Bernoulli probability measure $\eta$ on the Borel sets of $\Sigma_p^+$, then the measure $\eta$ portrays an asymmetric random walk on $G$ and it suggests a generalization of the concept of topological entropy of $S$.

\begin{definition}\label{de.relative-entropy}
Let $G$ be the semigroup generated by a set $G_1=\{id, g_1,\dots,g_p\}$, where each $g_i$ is a $C^2$ expanding map on a compact connected Riemannian manifold $M$, and let $S: G\times M \to M$ be the corresponding continuous semigroup action. Consider a probability measure $\eta$ on $\Sigma_p^+$. The \emph{relative topological entropy of the semigroup action $S$ with respect to $\eta$} is given by
$$h_{\text{top}}(S, \eta) =\lim_{\epsilon\to 0}\limsup_{n\to\infty}\frac1n\log
	\int_{\Sigma_p^+} \, s(g_{\omega_n} \dots g_{\omega_1}, n,\epsilon) \, d\eta(\omega)$$
where $s(g_{\omega_n} \dots g_{\omega_1}, n,\epsilon)$ denotes the maximum cardinality of a $(\underline g,n,\varepsilon)$-separated set (see Section \ref{sec:topol-entropy}) and $\omega=\omega_1 \omega_2 \ldots \omega_n \ldots$.
\end{definition}

The previous notion is well defined since the map $\omega \to s(g_{\omega_n} \dots g_{\omega_1}, n,\epsilon)$ is constant on $n$-cylinders (hence measurable), bounded by $e^{n\,\max_{i \in \{1,\ldots,p\}}\,\{h_{\text{top}}(g_i)\}}$ and $s(g_{\omega_n} \dots g_{\omega_1}, n,\epsilon)$ is monotonic in the variable $\epsilon$. For instance,
$h_{\text{top}}(S, \eta_{\underline p})= h_{\text{top}}(S).$
In view of Definition~\ref{de.relative-entropy}, $h_{\text{top}}(S, \eta)$ is also given by
$$h_{\text{top}}(S, \eta) =\limsup_{n\to\infty}\frac1n\log\Big(
	\sum_{|\underline g|=n} \iota_*(\eta)(\underline{g}) \, \mathcal{L}_{\underline g, 0}\, (\textbf{1})\, \Big).$$
Notice that this formula makes sense since $\underline{g} \mapsto \mathcal{L}_{\underline g, 0}\, (\textbf{1})$ is bounded, its values are away from $0$ and $\infty$ and $\iota_*(\eta)$ is a probability measure.

Following an argument analogous to the one used to prove \cite[Theorem 25]{RoVa1}, we obtain:

\begin{corollary}\label{cor:vep-generator-entropy}
Assume that the continuous action of $G$ on the compact metric space $M$ is strongly $\delta^*$-expansive and let $\eta\in \mathcal M(\Sigma_p^+)$ be a $\sigma$-invariant probability measure. Then, for every $0<\varepsilon<\delta^*$, we have
$$h_{\text{top}}(\eta,S)=
	\lim_{n\to\infty}\frac{1}{n}\log\int_{\Sigma_p^+}\,s(g_{\omega_n} \dots g_{\omega_1}, n,\epsilon)\,d\eta(\omega).$$
\end{corollary}

\subsubsection*{\textbf{\emph{A variational principle for the relative entropy}}}
Let $G$ be the semigroup generated by a set $G_1=\{id, g_1,\dots,g_p\}$, where each $g_i$ is a $C^2$ expanding map on a compact connected Riemannian manifold $M$, and let $S: G\times M \to M$ be the corresponding continuous semigroup action. Note that the action of $G$ on the compact metric space $M$ is strongly $\delta^*$-expansive. Denote by $\mathcal{M}_B(\Sigma_p^+)$ the space of Bernoulli measures on $\Sigma_p^+$, that is, the probability measures $\eta=\eta_{\underline a}$ for some probability vector $\underline a=(a_1, \dots, a_p)$, where some of the entries may be zero. This space of $\sigma$-invariant measures, which encodes all random walks on the semigroup $G$ we have considered so far, is homeomorphic to the finite dimensional simplex $\{\underline a=(a_1, \dots, a_p) \in \mathbb R^p \colon a_i\ge 0 \,\text{and}\, \sum_{i=1}^p a_i=1\}$, and therefore it is a closed subset of $\mathcal{M}_\sigma(\Sigma_p^+)$. Let $\mathcal{H}$ be the entropy map with respect to the random walks in $\mathcal{M}_B(\Sigma_p^+)$, given by
\begin{equation}\label{def:entropy-map}
\begin{array}{rccc}
  \mathcal{H}: & \mathcal{M}_B(\Sigma_p^+) &\to & [0,+\infty]  \\
\nonumber  & \eta &\mapsto & h_{\text{top}}(S, \eta).
\end{array}
\end{equation}

\begin{lemma}\label{continuous} The map $\mathcal{H}$ is continuous.
\end{lemma}

\begin{proof}
To prove that $\mathcal H$ is lower semicontinuous, we go back to the proof of Theorem~\ref{thm:A} where it was established that, if $\varepsilon \in (0,\delta_*)$, $T(\varepsilon/2)\in\mathbb N$ is given by the periodic specification property and $\underline{g}=g_{\omega_{m+n+T(\varepsilon/2)}} \dots g_{\omega_{1}} \in G^*_{m+n+T(\varepsilon/2)}$, then there exist
\begin{eqnarray*}
\underline a &=& g_{\omega_{m+n+T(\varepsilon/2)}}\dots g_{\omega_{m+T(\varepsilon/2)}+1} \in G_n^*\\
\underline b &=& g_{\omega_{m+T(\varepsilon/2)}} \dots g_{\omega_{m+1}} \in G^*_{T(\varepsilon/2)}\\
\underline c &=& g_{\omega_{m}} \dots g_{\omega_{1}} \in G^*_m
\end{eqnarray*}
such that $\underline g=\underline a \;\underline b\;\underline c$ and
$\sharp\,\mbox{Fix}(\underline{g})\geq \sharp\,\mbox{Fix}(\underline{a})\,\sharp\,\mbox{Fix}(\underline{c})$, 
and so
$$
\int \sharp\,\mbox{Fix}(\underline{g}) \; d\eta
        \geq\int \sharp\,\mbox{Fix}(\underline{c}) \;d\eta\;
	\int \sharp \,\mbox{Fix}(\underline{a})\;d\eta.
$$
In particular, as we are restricting to $\mathcal{M}_B(\Sigma_p^+)$, we get
\begin{equation}\label{eq:fixp}
\mathcal H(\eta) =\sup_{n\ge 1}\, \frac1n \,\log \,\int \sharp \mbox{Fix}\,(g_{\omega_n}\dots g_{\omega_1}) \; d\eta(\omega)
\end{equation}
is the supremum of the continuous functions $\eta \mapsto \frac1n\, \log\, \int \sharp \mbox{Fix}\,(g_{\omega_n}\dots g_{\omega_1}) \; d\eta$,
hence lower semicontinuous.

We proceed by showing that $\mathcal H$ is also upper semicontinuous. This is due to the characterization of of the topological entropy via generating sets. 
Indeed, the same steps of the proof  of Theorem 25 of \cite{RoVa1} (where integration is considered there with respect to the equidistributed Bernouli measure $\eta_p$) imply that, as
$S$ is a strongly $\delta^*$-expansive continuous semigroup action, then the topological entropy $h_{\text{top}}(S,\eta)$ satisfies
$$h_{\text{top}}(S,\eta) = \limsup_{n\to\infty}\frac{1}{n}\log\int_{\Sigma_p^+}c(g_{\omega_n} \dots g_{\omega_1}, n,\epsilon)d\eta(\omega)$$
for every $0<\varepsilon<\delta^*$, where 
$$\displaystyle c(g_{\omega_n} \dots g_{\omega_1}, n,\epsilon)=\inf\{\sharp\, \mathcal U:\mathcal U \text{ is a }(g_{\omega_n} \dots g_{\omega_1},\varepsilon)-\text{cover}\}.$$
Take $\varepsilon>0$, $n\in\NN$ and $\underline{g}=g_{\omega_{n+m}}\dots g_{\omega_1}\in G$, and consider
$\underline{\ell}=g_{\omega_{n+m}}\dots g_{\omega_{n+1}}$ and $\underline{k}=g_{\omega_{n}}\dots g_{\omega_1}$.
Given a $(\underline{\ell},\varepsilon)$-cover $\mathcal U$ and a $(\underline{k},\varepsilon)$-cover  $\mathcal V$,
we have that $\mathcal W=\underline k^{-1}(\mathcal U)\vee \mathcal V$ is a $(\underline{g},\varepsilon)$-cover with
$\sharp\mathcal W\leq\sharp\mathcal U\ \sharp\mathcal V$. This implies that
$$
\displaystyle c(g_{\omega_{n+m}} \dots g_{\omega_1}, n+m,\epsilon)
	\le c(g_{\omega_{n+m}} \dots g_{\omega_{m+1}}, m,\epsilon) \; c(g_{\omega_{n}} \dots g_{\omega_1}, n,\epsilon).
$$
Since all measures $\eta \in \mathcal{M}_B(\Sigma_p^+)$ are Bernoulli, 
the previous inequality yields 
\begin{align*}
\int_{\Sigma_p^+}c(g_{\omega_{n+m}} \dots g_{\omega_n}, n,\epsilon)d\eta(\omega)
       &\leq \int_{\Sigma_p^+}c(g_{\omega_{n}} \dots g_{\omega_1}, n,\epsilon)d\eta(\omega)
		\; \int_{\Sigma_p^+}c(g_{\omega_{m}} \dots g_{\omega_1}, m,\epsilon)d\eta(\omega).
\end{align*}
Hence, the sequence $\left(a_n\right)_{n \in \mathbb{N}}=\left(\log\int_{\Sigma_p^+}c(g_{\omega_{n}} \dots g_{\omega_1}, n,\epsilon)d\eta(\omega)\right)_{n \in \mathbb{N}}$ is subadditive
and so $\lim_{n\to\infty}\,\frac{a_n}{n} = \inf_{n\to\infty}\,\frac{a_n}{n}$. In other words,
$$
\mathcal H(\eta) =\inf_{n\ge 1} \,\frac1n\, \log \,\int c(g_{\omega_n}\dots g_{\omega_1}) \; d\eta(\omega)
$$
is the infimum of the continuous functions $\eta \mapsto \frac1n\, \log\, \int c(g_{\omega_n}\dots g_{\omega_1}) \; d\eta(\omega)$. Therefore, $\mathcal H$ is upper semicontinuous.
\end{proof}

\begin{proposition} There exists $\eta_0\in\mathcal{M}_B(\Sigma_p^+)$ such that
$$\sup_{\eta\in\mathcal{M}_B(\Sigma_p^+)}h_{\text{top}}(S, \eta)=h_{\text{top}}(S, \eta_0).$$
Moreover,
$$\sup_{\eta\in\mathcal{M}_B(\Sigma_p^+)}h_{\text{top}}(S, \eta) =\log \Big(\max_{1\le i \le p} \deg(g_i)\Big). $$ 
And $\sup_{\eta\in\mathcal{M}_B(\Sigma_p^+)}h_{\text{top}}(S, \eta)=h_{\text{top}}(S, \eta_p)$ if and only if
$\deg(g_i)=d$ for every $1\le i\le p$.
\end{proposition}

\begin{proof}
The first assertion is a direct consequence of the compactness of $\mathcal{M}_B(\Sigma_p^+)$ together with the
continuity of the function $\mathcal H$. We are left to prove that
\begin{equation}\label{eq:maxent}
\sup_{\eta\in\mathcal{M}_B(\Sigma_p^+)} h_{\text{top}}(S, \eta) =\log \Big(\max_{1\le i \le p} \deg(g_i)\Big).
\end{equation}
Let $1\le j \le p$ be such that $\deg(g_j)=\max_{1\le i \le p} \,\deg(g_i)$. Take $\underline a=(a_i)_{1\le i \le p}$, where $a_i=\delta_{ij}$
is the Kronecker delta function, and $\eta_{\underline a}=\delta_{jjj \dots}$. Then
$$
\sup_{\eta\in\mathcal{M}_B(\Sigma_p^+)} h_{\text{top}}(S, \eta)
	\ge h_{\text{top}}(S, \eta_{\underline a})
	= h_{\text{top}}(g_j)=\log \deg(g_j)
	= \log \Big(\max_{1\le i \le p} \deg(g_i)\Big).
$$
Conversely, assume that $\sup_{\eta\in\mathcal{M}_B(\Sigma_p^+)} h_{\text{top}}(S, \eta) > \log \Big(\max_{1\le i \le p} \deg(g_i)\Big)$. Using \eqref{eq:fixp}, we may find $\vep>0$ and $\eta\in\mathcal{M}_B(\Sigma_p^+)$ satisfying
\begin{equation*}
\lim_{n\to\infty} \frac1n \log \int \sharp \mbox{Fix}(g_{\omega_n}\dots g_{\omega_1}) \; d\eta(\omega)	> \log \Big(\max_{1\le i \le p} \deg(g_i)\Big) +2\vep.
\end{equation*}
Then, for every sufficiently large $n$, there exists $\omega \in \Sigma_p^+$, depending on $n$, such that
$$\sharp \,\mbox{Fix}\,(g_{\omega_n}\dots g_{\omega_1})> e^{\vep n} \,\Big(\max_{1\le i \le p} \deg(g_i)\Big)^n$$ 
which contradicts the growth rate of the fixed point sets of the expanding maps $g_{\omega_n}\dots g_{\omega_1}$.
Finally, recall that 
$$ h_{\text{top}}(S, \eta_p)=\log \Big(\frac{\sum_{i=1}^{p} \deg(g_i)}{p} \Big)$$
so $\eta_p$ is a maximizing measure for $\mathcal{H}$  
if and only if the maps $g_i$, for $1\le i \le p$, have equal degrees.
\end{proof}

\begin{remark}
It is clear from the equality in \eqref{eq:maxent} that any probability measure in $\mathcal{M}_B(\Sigma_p^+)$ that attains the maximum of $\mathcal{H}$ 
is of the form $\eta_{\underline a}\in \mathcal{M}_B(\Sigma_p^+)$ for some $\underline a$ in the simplex
$$
S=\Big\{ \underline a =(a_1. \dots, a_p) \in \mathbb R^+_0 : \sum_{i=1}^p a_i=1 \; \text{and}\; a_j=0 \,\text{whenever}\, \deg(g_j) \neq
\max_{1\le i \le p} \deg(g_i) \Big\}.
$$
In particular, uniqueness of such maximizing measures holds if and only if there exists a unique expanding map with larger degree.
\end{remark}

\subsubsection*{\textbf{\emph{Relative entropy \emph{vs.} annealed and quenched pressure}}}

In order to compare the relative entropies of the semigroup action with the several notions of pressure for skew-product dynamics we shall use the transfer operators defined in Section~\ref{sec:RPFoperators}. As previously remarked, when $\eta=\eta_{\underline a}$, we have
$$
h_{\text{top}}(S, \eta_{\underline a})
	=  \limsup_{n\to\infty}\frac1n \log \|\mathbf{\tilde L}_{\underline a, \underline \varphi}^n \,(\textbf{1})\|_0
	=  \limsup_{n\to\infty}\frac1n \log \|\mathbf{\hat L}_{\underline a, \underline \varphi}^n \,(\textbf{1})\|_0.
$$
It follows from \cite[Proposition~3.2]{Baladi} that the spectral radius of $\mathbf{\hat L}_{\underline a, \underline \varphi}$
coincides with the term $\exp(-P_{\text{top}}^{(a)}(\cF_G, \underline \varphi , \underline a) )$ and, for that reason,
\begin{equation}\label{eq:entropy-annealedpressure}
h_{\text{top}}(S, \eta_{\underline a}) = P_{\text{top}}^{(a)}(\cF_G, 0 , \underline a).
\end{equation}
This means that the relative entropy for asymmetric random walks coincides with the annealed topological
pressure and also implies the following generalization of formula (\ref{formula.Bufetov}).

\begin{corollary}\label{cor:Bufetov} Let $G_1=\{id, g_1,\cdots,g_p\}$, $p \geq 2$, be a finite set of expanding maps in $End^2(M)$. Given a non-trivial probability vector $\underline a=(a_1, a_2, \ldots,a_p)$, consider the Bernoulli probability measure $\eta_{\underline a}$ on $\Sigma_p^+$ and the annealed equilibrium state $\mu^{(a)}_{\underline a}$ for $\cF_G$ with respect to $\varphi \equiv 0$ and $\underline a$.
Then $h_{\text{top}}(S, \eta_{\underline a}) = P_{\text{top}}^{(a)}(\cF_G, 0 , \underline a).$
If $h^{\underline a}(\pi_{\mu^{(a)}_{\underline a}}) = h_{\pi_{\mu^{(a)}_{\underline a}}}(\sigma) - h_{\eta_{\underline a}}(\sigma)$, then
$$h_{\mu^{(a)}_{\underline a}}(\cF_G) = h_{\text{top}}(S, \eta_{\underline a}) + h_{\eta_{\underline a}}(\sigma).$$
\end{corollary}

\begin{proof} This is a direct consequence of the equality (\ref{eq:entropy-annealedpressure}) since, under the assumption on $h^{\underline a}(\pi_{\mu_{\underline a}})$, we have
$$P_{\text{top}}^{(a)}(\cF_G, 0 , \underline a) = \sup_{\{\mu \colon {\mathcal{F}_G}_*\mu=\mu\}} \Big\{ h_\mu(\cF_G) - h_{\pi_\mu}(\sigma) + h^{\underline a}(\pi_\mu)\Big\} = h_{\mu^{(a)}_{\underline a}}(\cF_G) - h_{\eta_{\underline a}}(\sigma).$$
Notice that the condition on $h^{\underline a}(\pi_{\mu_{\underline a}^{(a)}})$ is fulfilled when ${\underline a}= {\underline p}$ (cf. \eqref{eq:entropy-per-site-p}).
\end{proof}

Given a non-trivial probability vector $\underline a$, for some potentials $\varphi$ the transfer operator $\mathbf{\hat L}_{\underline a, \underline \varphi}$ coincides with the averaged normalized transfer operator used in \cite{S00}. Therefore, we may match the values of the corresponding pressures and their equilibrium states, and deduce the following thermodynamic criterium for the self-similar probability measures constructed in \cite{S00}.

\begin{proposition}\label{le:hfibered-a}
Let $G_1=\{id, g_1,\cdots,g_p\}$, $p \geq 2$, be a finite set of expanding maps in $End^2(M)$. Consider the semigroup $G$ generated by $G_1$ and  the corresponding skew-product $\mathcal{F}_G :  \Sigma_p^+  \times M  \to  \Sigma_p^+  \times M$. If $\underline a=(a_1, a_2, \ldots,a_p)$ 
is a non-trivial probability vector,  $\eta_{\underline a}$ the Bernoulli probability measure on $\Sigma_p^+$ determined by $\underline a$ and $\mu_{\underline a}$ the self-similar probability measure constructed in ~\cite{S00}, then the following assertions are equivalent:
\begin{enumerate}
\item $h_{\mu_{\underline a}}(\mathcal{F}_G) = \sup_{\{\mu \colon {\mathcal{F}_G}_*(\mu)=\mu,\,\pi_*(\mu)\,=\,\eta_{\underline a}\}} \,h_\mu(\mathcal{F}_G).$
\item $P_{\text{top}}^{(q)}(\cF_G, 0 , \underline a)
=  \int \log \deg \,(g_i)\,d\underline{a}(i).$
\end{enumerate}
\end{proposition}

We observe that, by \eqref{eq:quenched-a}, the condition (1) is equivalent to say that $\mu_{\underline a}=\mu^{(q)}_{\underline a}$, where $\mu^{(q)}_{\underline a}$ is the unique quenched equilibrium state of $\cF_G$ with respect to $\varphi \equiv 0$ and $\underline a$.

\begin{proof}
Fix $\underline a=(a_1, a_2, \ldots,a_p)$ and the potential $\underline \varphi =( -\log \deg \,(g_1), \dots, -\log \deg \,(g_p))$. The transfer operator $\mathbf{\hat L}_{\underline a, \underline \varphi}$ is precisely the averaged normalized transfer operator introduced in \cite{S00}. Therefore, $\mathbf{\hat L}_{\underline a, \underline \varphi} 1=1$, and consequently $P_{\text{top}}^{(a)}(\cF_G, \underline \varphi , \underline a)= 0$.
Moreover, $\mu_{\underline a}=\mu^{(a)}_{\varphi, \underline a}$, which is the unique annealed  equilibrium state for $\cF_G$ with respect to
$\underline \varphi$ and $\underline a$. So, equation~\eqref{def:annealed.entropy} yields
$$
h_{\mu_{\underline a}}(\mathcal{F}_G) =  - \int \log(\underline a(\omega_1) e^{\varphi(\omega,x)}) \, d\mu_{\underline a}.
$$
On the other hand, the quenched variational principle indicates that
$$
\sup_{\{\mu \colon {\mathcal{F}_G}_*\mu=\mu, \, \pi_\mu=\eta_{\underline a}\}} \Big\{ h_\mu(\cF_G) \Big\} - h_{\eta_{\underline a}}(\sigma)
	= P_{\text{top}}^{(q)}(\cF_G, 0 , \underline a).
$$
Thus the condition (1) in the statement of the proposition is equivalent to
\begin{align*}
P_{\text{top}}^{(q)}(\cF_G, 0 , \underline a)
	& = -\int \log(\underline a(\omega_1) e^{\varphi(\omega,x)}) \, d\mu_{\underline a} - h_{\eta_{\underline a}}(\sigma)\\
	& =  \sum_{i=1}^p - a_i \log a_i + \sum_{i=1}^p  a_i  \log \deg \,(g_i) +\sum_{i=1}^p  a_i \log a_i
	 =  \int \log \deg \,(g_i) d\underline{a}(i).
\end{align*}
\end{proof}

\begin{corollary}\label{cor:fibered-entropy-formula}
Let $G^*_1=\{ g_1,\cdots,g_p\}$, $p \geq 2$, be a finite set of $C^2$ expanding maps, $G$ the semigroup generated by $G_1$ and $\mathcal{F}_G :  \Sigma_p^+  \times M  \to  \Sigma_p^+  \times M$ the corresponding skew-product. Given a non-trivial probability vector $\underline a$ and the Bernoulli probability measure $\eta_{\underline a}$, 
we have
$$\int_{\Sigma_p^+}\,h_{\text{top}}(\mathcal{F}_G,\pi^{-1}(\omega))\,d\eta_a(\omega) = \sum_{k=1}^p\, a_k \,\log \deg\,(g_k)=\int_{\Sigma_p^+} \, \log \deg \,(g_{\omega_1}) \,d\eta_a(\omega).$$
\end{corollary}

\begin{proof}
We start noticing that, for any non-trivial probability vector $\underline a$ and its self-similar probability measure $\mu_{\underline a}$, the condition (1) in the statement of Proposition~\ref{le:hfibered-a} is valid (cf. \eqref{eq.sumi1}). Therefore, using Proposition~\ref{le:hfibered-a}, the equations (\ref{eq.sumi})
and (\ref{LW})
yield
$$\int_{\Sigma_p^+}\,h_{\text{top}}(\mathcal{F}_G,\pi^{-1}(\omega))\,d\eta_a = \sum_{k=1}^p\, a_k \,\log \deg\,(g_k).$$
The second equality in the statement is an immediate consequence from two facts: in each cylinder $C(1; k):=\{\omega \in \Sigma_p^+ \colon \omega_1=k\}$, the map $\log \deg \,(g_{\omega_1})$ is constant; and $\eta_a\,(C(1;k))=a_k$ for any $k \in \{1,2,\cdots,p\}$.
\end{proof}

\subsection{Examples}\label{sec:examples}

Let us analyze two (non-abelian) examples that illustrate the range of applications of our results on semigroup actions, transfer operators and the dynamical zeta function. 

\begin{example}\label{ex.2and3}

Let $g_1,\,g_2:\mathbb{S}^1\to\mathbb{S}^1$ be the circle expanding maps given by $g_1(z)=z^2$ and $g_2(z)=z^3$ and consider the semigroup $G$ generated by $G_1=\{id, g_1,g_2\}$. A simple computation shows that, for every $n \in \mathbb{N}$
$$N_n(G)=\frac1{2^n}\sum_{k=0}^{n} (2^k\,3^{n-k} -1)=\mathcal {O}\left(\frac{5}{2}\right)^n $$
and, consequently, $\wp(S) = \log \,(\frac{5}{2})$ and $\rho_S = \frac25$. Moreover, it follows from (\ref{formula.Bufetov}) that
$$h_{\text{top}}(S)=\log 5 - \log 2.$$

If $\eta_{\underline m}$ is the Bernoulli probability measure on $\Sigma_2^+$ determined by the weights $\underline m=(\frac{2}{5}, \frac{3}{5})\equiv (0.4, 0.6)$, equation (\ref{eq.sumi}) becomes
$$h_{\mu_{\underline m}}(\mathcal{F}_G) = h_{\text{top}}(\mathcal{F}_G) = \log 5$$
and equation (\ref{LW}) informs that
$$\int_{\Sigma_2^+}\,h_{\text{top}}(\mathcal{F}_G,\pi^{-1}(\omega))\,d\eta_{\underline m}(\omega)= \frac{2}{5}\,\log 2 + \frac{3}{5}\,\log 3.$$
So, $$h_{\text{top}}(S) < \int_{\Sigma_2^+}\,h_{\text{top}}(\mathcal{F}_G,\pi^{-1}(\omega))\,d\eta_{\underline m}(\omega).$$

Concerning $\underline a=(\frac{1}{2}, \frac{1}{2})\equiv (0.5, 0.5)$, which corresponds to the probability measure $\eta_{\underline a}=\eta_{\underline 2}$ on $\Sigma_2^+$, we deduce from Proposition~\ref{le:hfibered} that
$$h_{\text{top}}(S) > \int_{\Sigma_2^+}\,h_{\text{top}}(\mathcal{F}_G,\pi^{-1}(\omega))\,d\eta_{\underline 2}(\omega).$$
We may add that, if $\mu_{\underline a}$ is the self-similar measure previously mentioned that is assigned to $\underline a$ in \cite{S00}, then, again by (\ref{eq.sumi}) and (\ref{LW}), we have
$$h_{\mu_{\underline a}}(\mathcal{F}_G) = \log 2 + \frac{\log 2 + \log 3}{2}$$
and
$$\sup_{\mu \colon {\mathcal{F}_G}_*(\mu)=\mu,\,\pi_*(\mu)\,=\,\eta_{\underline 2}}\,h_\mu(\mathcal{F}_G)= \log 2 + \int_{\Sigma_2^+}\,h_{\text{top}}(\mathcal{F}_G,\pi^{-1}(\omega))\,d\eta_{\underline 2}(\omega).$$

\bigskip

\noindent Consequently,
$$\frac{\log 2 + \log 3}{2} \leq \int_{\Sigma_2^+}\,h_{\text{top}}(\mathcal{F}_G,\pi^{-1}(\omega))\,d\eta_{\underline 2}(\omega) < h_{\text{top}}(S)=\log 5 - \log 2.$$
Moreover, from Corollary~\ref{cor:fibered-entropy-formula}, we conclude that
$$\int_{\Sigma_2^+}\,h_{\text{top}}(\mathcal{F}_G,\pi^{-1}(\omega))\,d\eta_{\underline 2}(\omega)= \frac{\log 2 + \log 3}{2}$$
and, more generally, that, for any choice of $\underline a=(a_1,a_2)$ with $a_i>0$ and $a_1 + a_2 = 1$, we have
$$\int_{\Sigma_2^+}\,h_{\text{top}}(\mathcal{F}_G,\pi^{-1}(\omega))\,d\eta_{\underline a}(\omega)= a_1 \log 2 + a_2 \log 3.$$
Therefore, there is a (unique) vector $\underline a$ whose corresponding probability measure $\eta_{\underline a}$ on $\Sigma_2^+$ satisfies
$$h_{\text{top}}(S) = \int_{\Sigma_2^+}\,h_{\text{top}}(\mathcal{F}_G,\pi^{-1}(\omega))\,d\eta_{\underline a}(\omega),$$
namely
$$\underline a=\left(\frac{\log\,\frac{6}{5}}{\log\,\frac{3}{2}},\,\,\frac{\log\,\frac{5}{4}}{\log\,\frac{3}{2}} \right)\approx (0.45, 0.55).$$
\end{example}

\begin{example}
Given $p \in \mathbb{N}$, let $A_i\in GL(p,\mathbb Z)$ induce linear expanding endomorphisms $g_i:=g_{A_i}$ on $\mathbb T^p$, for $i=1,\dots, p$. Consider the continuous potential $\varphi:\Sigma^+_p \times M \rightarrow \mathbb{R}$ given by
$\varphi(\omega,x)= -\log |\det Dg_{\omega_1}(x)| = -\log \deg\,(g_{\omega_1}).$ Then, by \cite[Proposition~2]{Baladi}, the quenched and the annealed equilibrium states of $\varphi$ coincide and are SRB measures. In this setting, for any non-trivial probability vector $\underline a$, the self-similar $\cF_G$-invariant probability measure $\mu_{\underline a}$ constructed in \cite{S00} coincides with the annealed (hence quenched) SRB measure for $\cF_G$ with respect to $\cF$, $\underline \varphi$ and $\underline a$. 
In particular, we have
\begin{align*}
P_{\text{top}}^{(q)}(\cF_G, \underline \varphi , \underline a)
	=
h_{\mu_{\underline a}}(\cF_G) - h_{\eta_{\underline a}}(\sigma) + \int \varphi \, d\mu_{\underline a}
 \end{align*}
that is,
 $$
 h_{\mu_{\underline a}}(\cF_G)  = P_{\text{top}}^{(q)}(\cF_G, \underline \varphi , \underline a) + h_{\eta_{\underline a}}(\sigma)
 	- \int \varphi \, d\mu_{\underline a}
 $$
so, comparing this equality with \eqref{eq.sumi}, we obtain
$$P_{\text{top}}^{(q)}(\cF_G, \underline \varphi , \underline a) =  \sum_{i=1}^p\, a_i \,\log \deg\,(g_i)- \int \varphi \, d\mu_{\underline a}=\int\,\log \deg\,(g_i)\,d\underline a(i) - \int \varphi \, d\mu_{\underline a}.$$
\end{example}

\section{Stationary measures}\label{sec:stationary-measure}
Since probability measures invariant under all elements of the semigroup are unlikely to exist, the concept of stationary measure is the most natural to be addressed while studying ergodic properties of semigroup actions. Consider a finite set $G_1=\{id, g_1,\cdots,g_p\}$, $p \geq 2$, of $C^2$ expanding maps on a compact connected Riemannian manifold $M$ and let $G$ be the semigroup generated by $G_1$. Denote by $(\Sigma_p^+, \sigma)$ the full shift in $p$ symbols from the alphabet $\{1,\cdots,p\}$, by $\underline p$ the vector $(\frac{1}{p}, \ldots,\frac{1}{p})$ and by $\eta_{\underline p}$ the equally distributed Bernoulli $\sigma$-invariant probability measure on the Borel sets of $\Sigma_p^+$ given by the product measure of $\theta(\{i\})=\frac{1}{p}$ for any $i \in \{1,\cdots,p\}$.

Given a point $x \in M$ and $\omega=(\omega_1,\omega_2,\cdots) \in \Sigma_p^+$, we define the random orbit of $x$ as
$$x, \quad g_{\omega_1}(x), \quad g_{\omega_2}\,g_{\omega_1}(x),\ldots, g_{\omega_n}\,g_{\omega_{n-1}}\,\cdots g_{\omega_1}(x), \,\ldots.$$
This means that, for each $x \in M$ and $n \in \mathbb{N}$, the evolution of $x$ up to time $n$ is described by its images under the maps obtained by the concatenation of the $g_{\omega_i}$'s up to $\omega_n$. In what follows we will refer to this map as $f^n_\omega = g_{\omega_{n}}\,g_{\omega_{n-1}}\, \ldots \,g_{\omega_1}$. This way, the semigroup action of $G$ may be understood either as a random walk inside $End^2(M)$ or as a (non-local) random perturbation inside $End^2(M)$. In both readings, the orbits correspond to projections in the fiber $M$ of the orbits of the skew-product $\mathcal{F}_G$. More precisely, the shift $(\Sigma_p^+, \sigma)$ allows us to identify each element $i_n \dots i_1$ of the free semigroup $F_p$ and each $\underline{g}=g_{i_n}\,\cdots,\,g_{i_1}\in G$ satisfying $|\underline g|=n$ with $f^n_\omega$ for some suitable choice of $\omega \in \Sigma_p^+$: one must set $\omega_k=i_k$ for all $1 \leq k \leq n$. Observe, however, that for each concatenation $g_{i_n}\,\cdots,\,g_{i_1}$ we have several choices of paths $\omega$ where we may fit the finite word $i_1,\,\cdots,\,i_n$ in the first $n$ steps. These choices of $\omega$ define a cylinder in $\Sigma_p^+$, we denote by $C(1;i_1,\,\cdots,\,i_n)$, whose $\eta_{\underline p}$ measure is equal to $\frac{1}{p^n}$. We also notice that, although the $F_p$ consists of finite concatenations of generators, the induced skew-product $\mathcal{F}_G$, with a dynamics driven by the shift on one-sided sequences endowed with the invariant probability measure $\eta_{\underline p}$, spreads with equal probability to typical orbits with respect to the random walk $\iota_*(\eta_{\underline p})$ on $G$.

\begin{proposition}\label{prop-stationary}
Let $G$ be the semigroup generated by a set $G_1=\{id, g_1,\dots,g_p\}$, where each $g_i$ is a $C^2$ expanding map on a compact connected Riemannian manifold $M$, and let $S: G\times M \to M$ be the corresponding continuous semigroup action. Then there exists a H\"older-continuous nonnegative map $H_G:M \to \mathbb{R}$ such that $\nu_G = H_G\,\Leb$ is an absolutely continuous stationary probability measure for the semigroup $S$ and the symmetric random walk
$R_{\underline p}=\iota_*(\eta_{\underline p})$.
\end{proposition}

\begin{proof} The linear Koopman and transfer operators associated to each random perturbation of $f^1_\omega$ (which is $C^2$ expanding; see Remark~\ref{re.particular_cases}) are defined by
$$
U_\omega \,\psi = \psi\,(f^1_\omega)
\quad\text{and}\quad
\mathcal{L}_\omega \,(\psi)(x)= \sum_{f^1_\omega(y)=x}\,\frac{1}{|\det \,Df^1_\omega|}\, \psi (y)
$$
for any H\"older-continuous map $\psi$ and $x \in M$. Their iterates are, respectively,
$$
U_\omega^n \,\psi = \psi\,(f^n_\omega)
\quad\text{and}\quad
\mathcal{L}^n_\omega \,(\psi)(x)= \sum_{f^n_\omega(y)=x}\,\frac{1}{|\det \,Df^n_\omega|}\, \psi (y).
$$
Accordingly, the global transfer operators of the random perturbation of $f^1_\omega$ are given by
$$\widehat{U}_G \,\psi = \int\,U_\omega \,\psi\,\,d\eta_{\underline p}(\omega)
\quad\text{and}\quad
\widehat{\mathcal{L}}_G \,(\psi)(x) = \int\,\mathcal{L}_\omega \,(\psi) (x) \,d\eta_{\underline p}(\omega)
$$
for every Lebesgue-integrable observable $\psi:M\to\mathbb{R}$ and all $x \in M$.
Taking into account that, as $G_1$ is finite, then the expanding estimates of the elements of $G_1$ are uniform, we may apply the classical results concerning random perturbations of expanding dynamics and deduce that:

\begin{lemma}\cite[Section 2]{V97} There is a nonnegative H\"older function
$H_G=\lim_{n \to +\infty}\,\widehat{\mathcal{L}}_G^n \,(\textbf{1})$
which is a fixed point for the operator $\widehat{\mathcal{L}}_G$.
\end{lemma}

Normalizing $H_G$ so that $\int H_G\,d\Leb=1$, and setting $\nu_G=H_G\,Leb$, we obtain a probability measure in $M$ absolutely continuous with respect to Lebesgue.

\begin{lemma}
$\nu_G$ is a stationary measure.
\end{lemma}

\begin{proof}
The measure $\nu_G$ satisfies $\int\,(\widehat{U}_G \,\psi)\,d\nu_G=\int\,\psi\,d\nu_G$ for every $\psi \in C^0(M)$, that is,
$$\int\,\left(\int\,(\psi\circ g_{\omega_1})(x)\,d\eta_{\underline p}(\omega)\right) \,d\nu_G(x) = \int\,\psi(x)\,d\nu_G(x).$$
Using Fubini-Tonelli theorem and taking into account that $\eta_{\underline p}$ is equally distributed, we may exchange the order of integration on the first term and obtain
$$\int \Big( \int (\psi \circ \underline g)(x) \, d \nu_G(x) \Big) dR_{\underline p}(\underline g)=\int \psi (x)\,d\nu_G(x).$$ This confirms that $\nu_G$ is an $R_{\underline p}$-stationary measure for the semigroup $S$ (cf. (\ref{de.stationary})).
\end{proof}
\end{proof}

 In general one cannot expect the expanding maps of the semigroup to have common invariant probability measures; for that reason, the stationary probability measure is seldom invariant under all the elements of $G$. We also observe that, when studying the statistical stability of a dynamical system with respect to absolutely continuous invariant measures, one usually considers iterations of randomly  chosen dynamics in a neighborhood of the original dynamics and randomness is given by a parameter assumed to belong to an interval and to be absolutely continuous with respect to the Lebesgue measure (see for instance \cite{AAV} and references therein). In the previous result we did not require the $C^2$-expanding maps to be close to each other and the randomness of the iterations is given by the fixed random walk.

We note that the notions of stationary measure for the semigroup action $S$ and invariant measure for the skew-product $\mathcal{F}_G$ are somehow linked, as the next result indicates.

\begin{proposition}\label{le.stationary_invariant} Let $\eta$ be a $\sigma$-invariant probability measure on $\Sigma_p^+$ and consider $R_{\eta}=\iota_*(\eta)$. Then $\nu$ is an $R_{\eta}$-stationary measure on $M$ if and only if $\eta \times \nu$ is an $\mathcal{F}_G$-invariant probability measure.
\end{proposition}

\begin{proof} Given a continuous observable $\Psi : \Sigma_p^+  \times M \to \mathbb R$, we need to show that
$$\int (\Psi \circ \mathcal{F}_G)\, d(\eta \times \nu) = \int \Psi \,d(\eta \times \nu)$$
that is,
$$\int\int\, \Psi (\sigma(\omega), f^1_\omega(x)) \, d\eta(\omega)\,d\nu(x) = \int\int\, \Psi (\omega, x) \, d\eta(\omega) \,d\nu(x).$$
As $\nu$ is $R_{\eta}$-stationary, we have
$$\int \left(\int\,\Psi(\sigma(\omega), f^1_\omega(x))\,d\eta(\omega)\,\right) \, d\nu(x) =\int \left(\int\,\Psi(\sigma(\omega), x)\,d\eta(\omega)\right)\, d\nu(x).$$
Moreover, if we consider, for a fixed $x \in M$, the map
$$
\begin{array}{rccc}
V : & \Sigma_p^+ & \to & \mathbb{R} \\
	& \omega & \mapsto & \int\,\Psi(\omega,x)\,d\nu(x)
\end{array}
$$
then, as $\eta$ is $\sigma$-invariant and $V$ is continuous, we obtain
$$\int\, V(\sigma(\omega))\,d\eta(\omega)= \int\, V(\omega)\,d\eta(\omega)$$
that is,
$$\int \left(\int \,\Psi(\sigma(\omega),x)\,d\nu(x) \right)\,d\eta(\omega) = \int \left(\int\, \Psi(\omega,x)\,d\nu(x)\right)\,d\eta(\omega).$$
Therefore,
$$\int \left(\int\,\Psi(\sigma(\omega), f^1_\omega(x))\,d\eta(\omega)\,\right) \, d\nu(x) = \int \int \Psi(\omega,x)\, d\nu(x)\,d\eta(\omega).$$

Conversely, if $\mu$ is a probability measure invariant under the skew-product such that $\mu= (\pi_{\Sigma_p^+})_*(\mu) \times (\pi_M)_*(\mu)$, where $\pi_{\Sigma_p^+}$ and $\pi_M$ are the projections from $\Sigma_p^+ \times M$ onto the first and second coordinates, respectively, then clearly $(\pi_M)_*(\mu)$ is a $\iota_*((\pi_{\Sigma_p^+})_*(\mu))$-stationary measure.
\end{proof}

\section{Selection of measures for semigroup actions}\label{se.Therm-formalism}

The action of a semigroup generated by more than one dynamics is not a dynamical system, thus it is not straightforward how to define equilibrium states and establish a variational principle that might relate topological and measure theoretical aspects of the semigroup action. Yet, under adequate hypothesis, a semigroup action can be embodied into a dynamical system whose topological and measure theoretical properties we may study and convey to the semigroup action. 

From Section~\ref{se.entropy} recall that
$$
h_{\text{top}}(S, \eta_{\underline a})
	= P_{\text{top}}^{(a)}(\cF_G, 0 , \underline a)
	= P_{\text{top}}(\cF_G, \varphi_{\underline a})
$$
(cf. relations \eqref{eq:equalpress} and ~\eqref{eq:entropy-annealedpressure}). These two (different flavored) equalities justify the
construction of maximal entropy measures for semigroup actions, arising from skew-product dynamics,
that reflect the periodic data, the equidistribution among preimages
or both.

Given an H\"older potential $\psi: M \to \mathbb R$, consider the fiberwise constant potential in the skew product $\mathcal F_G$ defined as
$$
\begin{array}{rccc}
\varphi=\varphi_\psi : & \Sigma_p^+ \times M & \to & \mathbb R \\
	& (\omega,x) & \mapsto & \psi(x).
\end{array}
$$

\begin{definition}
We say that a probability measure $\nu$ on the Borel sets of $M$ is an \emph{equilibrium state for the semigroup action and $\psi$
(arising from the skew-product dynamics)} if $\nu=(\pi_M)_*(\mu_\varphi)$, where $\mu_\varphi$ is the unique equilibrium state for the
(topologically mixing Ruelle-expanding) map $\cF_G$ with respect to the fiberwise constant potential $\varphi_\psi$.
\end{definition}

\begin{remark}
If $\nu=(\pi_M)_*(\mu_\varphi)$ as in the previous definition and $\mu_\varphi=(\mu_\omega)_{\omega \in \Sigma_p^+}$
is a disintegration along the measurable partition $(\pi^{-1}(\omega))_{\omega \in \Sigma_p^+}$, guaranteed by Rohlin's theorem
then $\nu= \int_{\Sigma_p^+} \mu_\omega \, d\eta(\omega)$.
Although this resembles an $\eta$-stationarity condition  (recall $\nu$ is $\eta$-stationary if
$\nu = \int_{\Sigma_p^+} (g_{\omega})_* \nu \, d\eta(\omega)$) these do two notions do not coincide necessarily.
Indeed, if the semigroup $G$ is generated by $G_1=\{id, g_1, \dots, g_p\}$, where each $g_i$ is a $C^2$-smooth expanding map, and $\psi \equiv 0$, then $\mu_\varphi$ is the measure of maximal entropy of $\cF_G$ and it is also its annealed equilibrium state with respect to $\psi$ and $\eta_{\underline a}$, where $\underline{a} = \Big(\frac{\deg\,(g_1)}{\sum_{k=1}^p\,\deg\,(g_k)}, \frac{\deg\,(g_2)}{\sum_{k=1}^p\,\deg\,(g_k)},\ldots,\frac{\deg\,(g_p)}{\sum_{k=1}^p\,\deg\,(g_k)}\Big).$
Moreover, the probability $(\pi_M)_* (\mu_\varphi)$ is the measure constructed in \cite{Boy99}. Observe also that, from this reference, one can infer that the maximal entropy measure on $M$ is not, in general, a stationary measure.
\end{remark}

The discussion in Subsection~\ref{sec:random} suggests another approach. Given an H\"older continuous potential $\psi: M \to \mathbb R$ and the corresponding $\varphi=\varphi_\psi : \Sigma_p^+ \times M  \to  \mathbb R$, we may associate to any random walk on $G$ (determined by a probability measure $\eta_{\underline a}$ on $\Sigma_p^+$) two probability measures on $M$ which are the marginals on $M$ of 
the unique annealed and quenched equilibrium states $\mu^{(a)}_{\underline a,\varphi}$ and $\mu^{(q)}_{\underline a,\varphi}$, which may be distinct \cite[Section 2.4]{Baladi}. 
The projection $\pi_M : \Sigma_p^+ \times M \to M$ 
generates two measures on $M$, say $(\pi_M)_* (\mu^{(a)}_{\underline a,\varphi})$ and $(\pi_M)_* (\mu^{(q)}_{\underline a,\varphi})$, 
but observe that, even when the quenched and annealed states are different, it is not plain that they have different marginals on $M$. Since the physically observable measures are these $M$-marginals, it is worth exploring under what conditions on the semigroup action they coincide, in which case it would be natural to say that a probability measure $\nu$ on the Borel sets of $M$ is an \emph{equilibrium state for the semigroup action with respect to $\psi$ and $\underline a$} if $\nu=(\pi_M)_*(\mu^{(a)}_{\underline a,\varphi})$, where $\mu^{(a)}_{\underline a,\varphi}$ is the unique annealed (or quenched) equilibrium state for the skew-product $\cF_G$ with respect to $\varphi=\varphi_\psi$.

For instance, it may happen that the annealed equilibrium state $\mu^{(a)}_{\underline a,\varphi}$ for $\cF_G$ with respect to a potential $\varphi$  and a non-trivial probability vector $\underline a$ satisfies $\pi_{\mu^{(a)}_{\underline a,\varphi}} = \eta_{\underline a}$, in which case $h^{\underline a}(\pi_{\mu^{(a)}_{\underline a,\varphi}})=0$, and so $\mu^{(a)}_{\underline a,\varphi}$ is also the quenched equilibrium state of $\varphi$ with respect to $\underline a$ (cf. \cite[Proposition 2]{Baladi}). This occurs, for example, when $G$ is generated by $G_1=\{id, g_1, \dots, g_p\}$, where each $g_i$ is a $C^2$-smooth expanding map, and the potential $\varphi$ is defined as $\varphi(\omega,x)=-\log |\det Dg_{\omega_1}|(x).$ In this particular case, the marginal on $M$ of this common equilibrium has a disintegration in $M$ which is almost everywhere absolutely continuous with respect to the Lebesgue (cf. \cite[Remark~3.4]{Baladi}). Moreover, for the special vector $\underline p$, the operator
$$
\mathbf{L}_{n,\varphi}
	=\frac{1}{p^n}\sum_{|\underline{g}|=n} \mathcal{L}_{\underline g,\varphi}
	= \int_{\Sigma_p^+}  \mathcal{L}_{\underline g,\varphi} \, d\eta_{\underline p}(\omega)
$$
coincides with the averaged Ruelle-Perron-Frobenius of \cite{Baladi}. However, this potential is not $\varphi_\psi$ for any observable map $\psi$ on $M$.

\subsection*{Preimage equidistributed measures}
Consider a finite set $G_1=\{id, g_1,\cdots,g_p\}$, $p \geq 2$, of expanding maps in $End^2(M)$ and let $G$ be the semigroup generated by $G_1$. In analogy with the setting of a single Ruelle-expanding dynamical system (cf. Remark~\ref{re.particular_cases}), one expects a maximal entropy measure for a semigroup action to be computed as a weak$^{\ast}$ limit of a special sequence of probability measures on $M$.

\begin{example}
It follows from \cite{Boy99} that, when $\eta$ is the Bernoulli measure $\eta_{\underline p}$,
the sequence of measures
$$\frac{1}{\lambda^n}\,\sum_{|\underline{g}|=n} \,\sum_{\underline{g}(y)=x}\,\delta_y$$
is weak$^{\ast}$ convergent to some probability measure (independently of $x$).
Equivalently
\begin{equation}\label{eq:rmkboyd}
\frac{1}{e^{\log \frac{\lambda}p }} \,
		\Big[ \frac{1}{p^n}\,  \sum_{|\underline{g}|=n} \,\sum_{\underline{g}(y)=x}\,\delta_y \Big],
\end{equation}
where $\lambda = \deg\,(g_1) + \ldots + \deg\,(g_d)$. This special case, which corresponds to a symmetric random walk $\eta_{p}$, should be compared with the convergence of the equidistributed measures of the form~\eqref{eq:convmeasures} for the skew-product $\mathcal F_G$.
\end{example}

As noticed before, for $\eta=\eta_{\underline a}$ and an H\"older potential $\underline \varphi$ on $\Sigma^+_p \times M$, there is a natural stationary transfer operator $\mathbf{\tilde L}_{\underline a, \underline \varphi}$ acting on $C^0(M)$ as in \eqref{eq:RPF-integrated} and
$$
sp\,(\mathbf{\tilde L}_{\underline a, \underline \varphi})
	= sp\,(\mathbf{\hat L}_{\underline a, \underline \varphi})
	= \exp \,( P_{\text{top}}^{(a)}(\cF_G, \varphi , \underline a) ).
$$
Moreover, by \cite[Proposition~3.1]{Baladi}, the spectral features of the transfer operators
$\mathbf{\tilde L}_{\underline a, \underline \varphi}$
and $\mathbf{\hat L}_{\underline a, \underline \varphi}$ are strongly related.
For instance, $\lambda_{\underline a,\varphi}:= \exp \,( P_{\text{top}}^{(a)}(\cF_G, \varphi , \underline a) )$
is the leading eigenvalue for the transfer operator $\mathbf{\tilde L}^n_{\underline a, \underline \varphi}$ acting on
$C^r(M)$ ($r\ge 1$) with one-dimensional eigenspace generated by some $\rho \in C^r(M)$. The dual operator
$\mathbf{\tilde L}^{*}_{\underline a, \underline \varphi}$ defined, for every continuous $\psi : M \to \mathbb R$ by
$$
\int \psi \; d\mathbf{\tilde L}^{*}_{\underline a, \underline \varphi} \eta
	= \int \mathbf{\tilde L}_{\underline a, \underline \varphi} \psi \; d \eta
$$
has also a one-dimensional eigenspace associated to the leading eigenvalue $\lambda_{\underline a,\varphi}$
generated by some probability measure $\gamma$ on $M$.
Moreover, $(\pi_M)_* \circ \mathbf{\hat L}^*_{\underline a, \underline \varphi}
= \mathbf{\tilde L}^*_{\underline a, \underline \varphi} \circ (\pi_M)_*$ and for each $x\in M$
the measures
\begin{equation}\label{eq:integral}
\lambda_{\underline a,\varphi}^{-n}\,\, (\mathbf{\tilde L}^{n}_{\underline a, \underline \varphi})_* \,\delta_x
	= \frac{1}{\lambda_{\underline a,\varphi}^n}\, \int
	\sum_{g_{\omega_n} \dots g_{\omega_1} (y)=x} \, e^{S_n \varphi_{\underline g}(y)} \, \delta_y
		\; d\eta_{\underline a} ( \omega_1, \dots, \omega_n )
\end{equation}
obtained by averaging the preimages of $x$ according to the random walk $\eta_{\underline a}$ are convergent to
the measure $\rho(x) \cdot \gamma$ on $M$.

On the other hand, the annealed equilibrium state $\mu^{(a)}_{\varphi,\underline a}$ on $\Sigma_p^+\times M$
is absolutely continuous with respect to
a conformal measure for $\mathbf{\hat L}^*_{\underline a, \underline \varphi}$: \,there exists a probability measure $\hat \gamma$
on $\Sigma_p^+ \times M$ so that $\mathbf{\tilde L}^{*}_{\underline a, \underline \varphi} \hat\gamma= \lambda_{\underline a,\varphi} \, \hat\gamma$,
that $\mu^{(a)}_{\varphi,\underline a} = \rho \hat \gamma$ and $(\pi_M)_* \hat\gamma=\gamma$ (cf. \cite[Proposition 3.1(2) and Proposition 3.2]{Baladi}).
Thus, it is natural to consider the marginal measure of $\mu^{(a)}_{\varphi,\underline a}$,
$$
\nu_{\varphi, \underline a}:=(\pi_M)_*\mu^{(a)}_{\varphi,\underline a}= \rho \gamma,
$$
on $M$, which is a probability measure.

Given a random walk on $G$ (determined by a $\sigma$-invariant probability measure on $\Sigma_p^+$), the previous information on the integrated and fibered transfer operators legitimates the following concept of maximal entropy measure for the semigroup action with respect to the fixed random walk. Another approach may be found in \cite{GLW}, although not expressing the variational connections in terms of invariant measures; see ~\cite{BisII} for details.

Notice that the major advantage of this definition is that \eqref{eq:integral} defines $\nu_{\varphi, \underline a}$ intrinsically, using the generators of the semigroup. 
Using \eqref{eq:quenched-annealed}, the next proposition summarizes the main link maximal entropy measures for the semigroup $S$ and for the skew-product.

\begin{proposition}\label{prop:mme.preimages}
Let $G$ be the semigroup generated by a set $G_1=\{id, g_1,\dots,g_p\}$ of $C^2$ expanding maps on a compact connected Riemannian manifold $M$. Consider the continuous semigroup action $S: G\times M \to M$ and the random walk $R_{\underline a}=\iota_*(\eta_{\underline a})$. Then there is a probability measure $\nu_{0, \underline a}$ of maximal entropy for the semigroup $S$ with respect to the random walk $R_{\underline a}$ (arising from the skew product dynamics) and
\begin{equation*}\label{eq:mme.preimages}
\nu_{0, \underline a} =(\pi_M)_*(\mu^{(a)}_{\underline a}).
\end{equation*}
Moreover, if $\underline a=\underline m$, then $\mu^{(q)}_{\underline m}=\mu_{\underline m}= \mu^{(a)}_{\underline p}$ and the following
equalities hold:
$$
\nu_{0, \underline a} = (\pi_M)_*(\mu_{\underline m})=(\pi_M)_*(\mu^{(a)}_{\underline p})=(\pi_M)_*(\mu^{(q)}_{\underline m}).
$$
\end{proposition}

This ends the proof of Theorem~\ref{thm:C}. One should mention that some similar expressions can be written in the case of locally constant potentials. Nevertheless we shall not need or use this fact here.

\subsection*{Questions} Is there a generalized version of the previous proposition for other potentials $\varphi$? What if $\eta=\eta_{\underline a}$ with ${\underline a} \neq {\underline p}$? Recall also that $\nu$ is an $\eta$-stationary measure if $\nu = \int_{\Sigma_p^+} (g_{\omega})_* \nu \, d\eta(\omega)$, a condition that resembles \eqref{eq:integral}: how do these properties relate? Are there equilibrium states that do not arise from skew-product dynamics? Do all equilibrium measures have some Gibbs property?

\subsection*{Periodic masses convergence}
The previous Proposition~\ref{prop:mme.preimages} establishes the existence of a maximal entropy measure for the semigroup
action $S$ that arises from skew-product dynamics and can be intrinsically defined by a process of averaging preimages
(cf. \eqref{eq:integral}). It is a relevant issue to understand if such maximal entropy measures can also arise from
the periodic points for the finite time sequential dynamics.

As mentioned before, the skew-product $\cF_G$ is a Ruelle expanding map,
$h_{\text{top}}(S, \eta_{\underline a}) = P_{\text{top}}(\cF_G, \varphi_{\underline a})$ and
$\mu_{\underline m}= \mu^{(a)}_{\underline p}$. Consequently, we obtain the following:

\begin{proposition}
Let $G$ be the semigroup generated by a set $G_1=\{id, g_1,\dots,g_p\}$ of $C^2$ expanding maps on a compact connected Riemannian manifold $M$. Consider the continuous semigroup action $S: G\times M \to M$ and the symmetric random walk $R_{\underline p}=\iota_*(\eta_{\underline p})$. Then there is a probability measure $\nu_{\varphi, \underline p}$ of maximal entropy for the semigroup $S$ with respect to the random walk $R_{\underline p}$ (arising from the skew product dynamics) and
\begin{align*}\label{eq:mme.preimages}
\nu_{0, \underline p}
	& = \lim_{n\to \infty} e^{-n \, h_{top}(S,\eta_{\underline p})}
			\sum_{ \sigma^n(\omega)=\omega} \,\, \sum_{g_{\omega_n} \dots g_{\omega_1}(x)=x} \delta_x .
\end{align*}
In particular,
$$
\nu_{0, \underline p} (A)
	 = \lim_{n\to \infty} e^{-n \, h_{top}(S,\eta_{\underline p})}
		\sum_{ \sigma^n(\omega)=\omega}  \# \big \{  \text{Fix} (g_{\omega_n} \dots g_{\omega_1}) \cap A \big\}
$$
for any measurable set $A\subset M$ satisfying $\nu_{0, \underline p} (\partial A)=0$.
\end{proposition}

\begin{proof}
In this setting, $\mu^{(a)}_{\underline p}=\mu_{\underline m}$. By the continuity of the push-forward map $(\pi_M)_*$ and
the asymptotic growth rate of periodic orbits, we conclude that
\begin{align*}
\nu_{0, \underline p}
	& = (\pi_M)_*(\mu^{(a)}_{\underline p})
	= (\pi_M)_*(\mu_{\underline m})
	= (\pi_M)_*\Big( \lim_{n\to \infty} \frac{\sum_{ (\omega,x) \in \text{Fix} \; (\mathcal F_G^n)} \delta_{(\omega,x)} }
			{\# \text{Fix} \; (\mathcal F_G^n)} \Big) \\
	& =  \lim_{n\to \infty} (\pi_M)_*\Big(
			\frac{\sum_{ \sigma^n(\omega)=\omega}  \sum_{g_{\omega_n} \dots g_{\omega_1}(x)=x} \delta_{(\omega,x)} }
			{\sum_{\sigma^n(\omega)=\omega}  \#\text{Fix}(g_{\omega_n} \dots g_{\omega_1}) } \Big) \\
	& =   \lim_{n\to \infty} e^{-n h_{top}(S, \,\eta_{\underline p})}
			\sum_{ \sigma^n(\omega)=\omega}  \sum_{g_{\omega_n} \dots g_{\omega_1}(x)=x} \delta_x.
\end{align*}
The second assertion is immediate from the definition of the weak$^{\ast}$ convergence.
\end{proof}

\subsection*{Acknowledgements}
FR has been financially supported by BREUDS. PV has benefited from a fellowship awarded by CNPq-Brazil and is grateful to the Faculty of Sciences of the University of Porto for the excellent research conditions. MC has been financially supported by CMUP (UID/MAT/00144/2013), which is funded by FCT (Portugal) with national (MEC) and European structural funds through the programs FEDER, under the partnership agreement PT2020.

\end{document}